\documentclass[oneside]{amsart}
\usepackage{geometry}
 \usepackage{graphicx}
 \usepackage{amssymb}
 \usepackage{comment}
 \usepackage[mathscr]{euscript}
 \usepackage{tikz}
\usetikzlibrary{decorations.pathmorphing, patterns,shapes,arrows,automata,positioning}
\usepackage{amsmath}
\usepackage{amsfonts}
\usepackage{mathtools}
\usepackage{amsthm}
\usepackage{hyperref}
\usepackage{enumerate}
\usepackage{indentfirst}
\usepackage{tikz}
\usetikzlibrary{arrows,automata,positioning}
\usepackage{color}

\newtheorem{thm}{Theorem}

\newtheorem{lem}[thm]{Lemma}
\newtheorem{fact}[thm]{Fact}
\newtheorem{cor}[thm]{Corollary}

\newtheorem*{thm*}{Theorem}
\newtheorem*{thmA}{Theorem A}
\newtheorem*{thmB}{Theorem B}
\newtheorem*{thmC}{Theorem C}
\newtheorem*{thmD}{Theorem D}
\numberwithin{thm}{section}

\theoremstyle{definition}
\newtheorem{defn}[thm]{Definition}
\newtheorem*{defn*}{Definition}
\newtheorem*{ack}{Acknowledgements}

\newcommand{\N}{\mathbb{N}}
\newcommand{\R}{\mathbb{R}}

\newcommand{\Z}{\mathbb{Z}}
\newcommand{\Q}{\mathbb{Q}}

\newcommand{\calR}{\mathcal{R}}

\newcommand{\calA}{\mathcal{A}}
\newcommand{\calB}{\mathcal{B}}
\newcommand{\calC}{\mathcal{C}}

\newcommand{\Diam}{\operatorname{Diam}}
\newcommand{\sprad}{\operatorname{sprad}}
\begin{document}

\author{Alexi Block Gorman}
\address
{The Fields Institute \\222 College Street\\Toronto, ON M5T 3J1\\Canada}
\email{ablockgo@fields.utoronto.ca}

\author{Christian Schulz}
\address
{Department of Mathematics\\University of Illinois at Urbana-Champaign\\1409 West Green Street\\Urbana, IL 61801}
\email{cschulz3@illinois.edu}


\title{Fractal dimensions of $k$-automatic sets}

\begin{abstract}
This paper seeks to build on the extensive connections that have arisen between automata theory, combinatorics on words, fractal geometry, and model theory. 
Results in this paper establish a characterization for the behavior of the fractal geometry of ``$k$-automatic'' sets, subsets of $[0,1]^d$ that are recognized by B\"uchi automata. 
The primary tools for building this characterization include the entropy of a regular language and the digraph structure of an automaton.
Via an analysis of the strongly connected components of such a structure, we give an algorithmic description of the box-counting dimension, Hausdorff dimension, and Hausdorff measure of the corresponding subset of the unit box.
Applications to definability in model-theoretic expansions of the real additive group are laid out as well.
\end{abstract}

\maketitle

\section{Introduction}
\subsection{Main Results}

In this paper, we consider the $k$-automatic subsets of $\R$, and analyze both the $k$-representations of such sets and the B\"uchi automata that recognize their base-$k$ representations.
The methods used in this paper integrate multiple perspectives previously taken regarding B\"uchi automata and $k$-automatic subsets of finite-dimensional Euclidean spaces.
These include the perspective given by viewing automata as directed graphs, as well as characterizations of $k$-regular $\omega$-languages coming from combinatorics on words.

Our primary result describes how to obtain the Hausdorff and box-counting dimensions of $k$-automatic subsets of $[0,1]^d \subseteq \R^d$ (with $d \in \N$) not quite in terms of some of the induced subautomata, but by considering slight variants thereof.
Further results include a similar mechanism for computing Hausdorff measure (in the appropriate dimension) in terms of the same variant of an induced subautomaton, as well as a characterization of which expansions of the first order structure $(\R,<,+)$ by $k$-automatic subsets of $[0,1]^d$ have definable unary sets whose Hausdorff dimension differs from its box-counting dimensions.

Recall that an automaton is ``trim'' if each state is accessible from some start state, and each state is also coaccessible from some accept state.
Below, we use $d_H$ to denote Hausdorff dimension, we use $d_B$ to denote box-counting dimension, and $h(X)$ denotes the entropy of $X$; for formal definitions of each, see section \ref{prelim}.
The following theorem is in fact a corollary in section \ref{closed} that emerges from unpacking results of Mauldin and Williams in \cite{MW88}, in conjunction with some elementary results in section \ref{entropy}.

\begin{thmA}
If $X$ is a closed $k$-automatic subset of $[0,1]^n$ recognized by closed automaton $\calA$, then $d_H(X) = d_B (X)= \frac{1}{\log (k)} h(L(\calA))$, where $L(\calA)$ is the set of strings $\calA$ recognizes.
\end{thmA}

To state our main theorem, we will briefly describe the ``cycle language'' associated to a state $q$ in an automaton $\calA$, the definition of which is stated with more detail in section \ref{dimensions}.
Suppose that $\calA$ is an automaton (finite or B\"uchi) with $Q$ as its set of states.
For state $q \in Q$, the cycle language $C_q(\calA)$ is the language consisting of all $w \in \Sigma^*$ such that there is a run of $\calA$ from state $q$ to itself via $w$.
\begin{thmB}
    Let $\mathcal{A}$ be a trim B\"uchi automaton with set of states $Q$, and let $X$ be the set of elements in $[0,1]^d \subseteq \R^d$ that have a base-$k$ representation that $\calA$ accepts. Let $F$ be the set of accept states of $\mathcal{A}$. Then:
	
	\begin{enumerate}[(i)]
		\item $d_H(X) = \frac{1}{\log k} \max_{q \in F} h(C_q(\mathcal{A}))$;
		\item $d_B(X) = \frac{1}{\log k} \max_{q \in Q} h(C_q(\mathcal{A}))$.
	\end{enumerate}
\end{thmB}

From this theorem, and using the crucial notion of ``unambiguous'' B\"uchi automata, we establish a similar result that describes the Hausdorff measure of a $k$-automatic set in terms of the structure of the automaton that recognizes it.
For a definition of unambiguous automata and details about the partition $\{M_q:q\in Q'\}$ of the language $L$ in the theorem below, see Section \ref{measure}.

\begin{thmC}
Let $\mathcal{A}$ be an unambiguous B\"uchi automaton with set of states $Q$ and recognizing an $\omega$-language $L$.
Let $Q' \subseteq Q$ be the set of states whose strongly connected component contains an accept state. 
For each $q \in Q'$, let $\mathcal{A}_q$ be the automaton created by moving the start state of $\mathcal{A}$ to $q$ and removing all transitions out of its strongly connected component, and let $L_q$ be the $\omega$-language it accepts. 
Then we can effectively partition $L$ into sublanguages $\{M_q:q\in Q'\}$ such that:

	\begin{enumerate}[(i)]
		\item $d_H(\nu_k(L)) = \max_{q \in Q'} d_H(\nu_k(L_q))$,
		\item with $\alpha = d_H(\nu_k(L))$, $\mu_H^{\alpha}(L) = \sum_{q \in Q'} \mu_H^{\alpha}(M_q)$.
	\end{enumerate}
\end{thmC}

Finally, we give a dividing line for the fractal dimensions of definable sets in certain first-order sets related to B\"uchi automata.
This dividing line has implications for the model-theoretic tameness of structures of the form $(\R,<,+,X)$ where $X \subseteq [0,1]^n$ is $k$-automatic, since Hieronymi and Walsberg have shown in \cite{HW19} that if $\calC$ is a Cantor set (compact, has no interior and no isolated points) then $(\R,<,+,\calC)$ is not tame with respect to any notion coming from Shelah-style generalizations of stability, including NIP and NTP$_2$.

\begin{thmD}
Suppose $X \subseteq [0,1]^n$ is $k$-automatic.  There exists a set $A \subseteq [0,1]$ definable in $(\R, <,+, X)$ such that $d_B(A) \neq d_H(A)$ if and only if either a unary Cantor set is definable in $(\R,<,+,X)$ or a set that is both dense and codense on an interval is definable in $(\R,<,+,X)$.
\end{thmD}

\subsection{Background}
In his seminal work \cite{B62}, J. R. B\"{u}chi introduced the notion of what we now call a B\"{u}chi automaton, and he identified a connection between these automata and the monadic second-order theory of the natural numbers with the successor function.
Notably, B\"{u}chi automata take countably infinite-length inputs, unlike standard automata (which we will also call ``finite automata''), which only accept or reject finite-length input strings.
In addition to the work of B\"{u}chi to extend the notion of automatic sets to infinite words, in \cite{M66} McNaughton broadened the realm of generating infinite sequences by a finite automaton, and many more notions of automatic or regular sets of infinite words arose.

There is natural topological structure on the space of infinite words on a finite alphabet, 
hence the topological features of subsets of such a space that is recognized by an appropriate B\"{u}chi automaton have been investigated since the 1980s.
Languages recognized by B\"{u}chi automata are commonly called regular $\omega$-languages.
One topological property that was first introduced in the context of information theory by Claude Shannon in \cite{S48} is that of entropy, also called ``topological entropy'' in some settings.

In \cite{S85}, L. Staiger established that extending the definition of entropy to $\omega$-languages yields compelling topological characterizations of closed regular $\omega$-languages.
For example, he shows in \cite{S85} that a closed regular $\omega$-language is countable if and only if the entropy is $0$ and that the entropy of regular $\omega$-languages is countably additive.
From another perspective, it was shown in \cite{CLR15} that there is a close connection between regular $\omega$-languages and Graph Directed Iterated Function Systems, or GDIFSs for brevity.
Due to the work in \cite{MW88} there have long existed means of computing geometric properties like Hausdorff measure and Hausdorff dimension for GDIFSs.
In light of the connection between B\"{u}chi automata and GDIFSs, the connections between fractal dimensions and entropy for automatic sets of real numbers can now yield a total characterization, as our paper illuminates.

\begin{ack}
    Many thanks to Philipp Hieronymi for the interesting ideas, questions, and generous guidance provided for this paper.
    Thanks to Elliot Kaplan, Jason Bell, and Rahim Moosa, for their helpful questions and comments regarding the contents of this paper.
    We gratefully acknowledge that this research was supported by the Fields Institute for Research in Mathematical Sciences. Its contents are solely the responsibility of the authors and do not necessarily represent the official views of the Institute.
    The first author was partially supported by the National Science Foundation Graduate Research Fellowship Program under Grant No. DGE -- 1746047. The second author was partially supported by National Science Foundation Grant no. DMS -- 1654725.
\end{ack}

\section{Preliminaries}\label{prelim}

\subsection{Definition of B\"uchi automata}

Below, for a set $X$ we use $X^*$ to denote the Kleene star of $X$, i.e. $X^*=\{x_1x_2\ldots x_n: n\in \N, x_1,\ldots ,x_n\in X\}$, and we similarly use $X^{\omega}$ to denote the set $\{x_1x_2\ldots: x_1, x_2, \ldots \in X\}$. 
For a language $L$ of finite strings, we will use $\vec{L}$ to denote the limit language of $L$, i.e. the set of infinite strings with infinitely many prefixes in $L$.

\begin{defn}
    A \textit{finite automaton} is a $5$-tuple $\mathcal{A} = (Q, \Sigma, \delta, S, F)$ where:
    
    \begin{itemize}
        \item $Q$, the set of \textit{states,} is a finite set;
        \item $\Sigma$, the \textit{alphabet,} is a finite set;
        \item $\delta$, the \textit{transition function,} is a function $Q \times \Sigma \to \mathcal{P}(Q)$;
        \item $S$, the set of \textit{start states} or \textit{initial states,} is a nonempty subset of $Q$;
        \item $F$, the set of \textit{accept states} or \textit{final states,} is a subset of $Q$.
    \end{itemize}
    
    A finite automaton is said to \textit{run from $q_0$ to $q_n$ on} a string $w = w_1 \dots w_n \in \Sigma^n$, for $q_0, q_n \in Q$, 
    if there exist states $q_1, \dots, q_{n-1}$ such that for $i = 1, \dots, n$ we have $q_i \in \delta(q_{i-1}, w_i)$. 
    If $q_0 \in S$, 
    such a sequence of states may be called a \textit{run} of $w$ in $\mathcal{A}$, which is \textit{accepting} if $q_n \in F$. 
    The automaton \textit{accepts} $w$ if there is an accepting run of $w$. 
    The language \textit{recognized} (or accepted) by $\mathcal{A}$ is the set of all strings in $\Sigma^*$ it accepts. 
    Two finite automata are \textit{equivalent} if they recognize the same language.
\end{defn}

\begin{defn}
    A \textit{B\"uchi automaton} is a $5$-tuple $\mathcal{A} = (Q, \Sigma, \delta, S, F)$ where:
    
    \begin{itemize}
        \item $Q$, the set of \textit{states,} is a finite set;
        \item $\Sigma$, the \textit{alphabet,} is a finite set;
        \item $\delta$, the \textit{transition function,} is a function $Q \times \Sigma \to \mathcal{P}(Q)$;
        \item $S$, the set of \textit{start states} or \textit{initial states,} is a nonempty subset of $Q$;
        \item $F$, the set of \textit{accept states} or \textit{final states,} is a subset of $Q$.
    \end{itemize}
    
    For an infinite string $w = w_1 w_2 \dots \in \Sigma^\omega$, a \textit{run} of $w$ in $\mathcal{A}$ is a sequence of states $q_0, q_1, \dots \in Q^\omega$ such that $q_0 \in S$ and for $i \in \mathbb{Z}^+$ we have $q_i \in \delta(q_{i-1}, w_i)$. 
    A run is \textit{accepting} if $q_i \in F$ for infinitely many $i$. 
    The automaton \textit{accepts} $w$ if there is an accepting run of $w$. 
    The $\omega$-language \textit{recognized} (or accepted) by $\mathcal{A}$ is the set of all strings in $\Sigma^\omega$ it accepts. 
    Two B\"uchi automata are \textit{equivalent} if they recognize the same language.
\end{defn}

Note that the only difference between these definitions is in the accept condition; thus, the same tuple $(Q, \Sigma, \delta, S, F)$ may be alternately treated as either a finite or B\"uchi automaton, which will be useful several times in this paper. 
A finite or B\"uchi automaton also has a canonical digraph structure whose vertex set is $Q$ and whose edge set contains precisely those $(q, q') \in Q^2$ for which there exists $\sigma \in \Sigma$ such that $q' \in \delta(q, \sigma)$. 
We will often implicitly refer to this digraph structure, speaking of such concepts as paths between states and strongly connected components containing states.
If we refer to the graphical structure on an automaton as simply a graph, we implicitly mean the structure of the automaton as a directed graph.

We will also use several properties that such an automaton may have:

\begin{defn} \label{autoproperties}
Let $\mathcal{A} = (Q, \Sigma, \delta, S, F)$ be a finite or B\"{u}chi automaton.
    \begin{enumerate}[(i)]
        \item We say $\calA$ is \textit{deterministic} if $|S| = 1$ and $|\delta(q, c)| \leq 1$ for all $q \in Q, c \in \Sigma$. (Note that this definition guarantees that there is at most one run of a given $w$ in $\mathcal{A}$.) Every finite automaton has an equivalent deterministic automaton; this is not true in general for B\"uchi automata.
        \item We say $\calA$ is \textit{finite-trim} if for every $q \in Q$, there is a path from a start state to $q$ (possibly of zero length) and a path from $q$ to an accept state (also possibly of zero length). On the additional condition that the path from $q$ to an accept state must be of nonzero length, we say that $\calA$ is \textit{trim}. Every B\"uchi automaton has an equivalent trim automaton; every finite automaton has an equivalent finite-trim automaton. In fact, given an automaton $\calA$, we may always produce a finite-trim automaton that is equivalent to $\calA$ as both a finite and B\"uchi automaton.
        \item We say $\calA$ is \textit{closed} if it is trim and every state is an accept state (i.e., $Q=F$).
        \item Given a trim automaton $\calA=(Q,\Sigma,\delta,S,F)$, call $\overline{\calA}=(Q,\Sigma,\delta,S,Q)$ (this is the resulting automaton when all the states of $\calA$ are added to the set of accept states) the \textit{closure} of $\calA$.
        Note that if an automaton $\calB=(Q',\Sigma, \delta',S',F')$ is equivalent to $\calA$ but not trim, then $(Q',\Sigma,\delta',S',Q')$ need not recognize the same language as $\overline{\calA}$.
        \item We say an automaton $\calA=(Q,\Sigma,\delta,S,F)$ is \textit{weak} if for every $q,q' \in Q$ such that $q$ and $q'$ are in the same strongly connected component of $\calA$ (as a digraph), either $q$ and $q'$ are both accept states, or both are not accept states.
    \end{enumerate}
\end{defn}

\subsection{Regularity and $k$-representations}
Let $k \in \N_{>1}$, and set $[k]=\{0,1, \ldots ,k-1\}$ for the remainder of this paper.
We will use the terms ``base-$k$ representation'' and ``$k$-representation'' interchangeably to mean the expression of an element $x \in \R$ as a countable sum of integer powers of $k$, each multiplied by a coefficient in  $[k]$.
Note that we will sometimes conflate elements of $[0,1]$ and their $k$-representations, and we may occasionally say that an automaton $\calA$ accepts \emph{the} $k$-representation of $x \in [0,1]$.
For the countable subset of $[0,1]^d$ whose elements have multiple (in particular, at most $2^d$) $k$-representations, we mean that $\calA$ accepts \emph{at least one} of the $k$-representations of $x$.
For ease of switching between $x$ and its $k$-representation, we will define a valuation for elements of $[k]^{\omega}$.

\begin{defn}
Define $\nu_k:[k]^{\omega} \to [0,1]$ by:
$$\nu_k(w) = \sum_{i=0}^{\infty} \frac{w_i}{k^{i+1}}$$
where $w=w_0w_1w_2\ldots$ with $w_i \in [k]$ for each $i \in \N$.
\end{defn}

Note that the equivalence relation $v \equiv w \iff \nu_k(v) = \nu_k(w)$ is not only a finite equivalence relation, but moreover each equivalence class has size at most two.
As noted above, only countably many elements in $[k]^{\omega}$ are not the unique element of their $\nu_k$-equivalence class.
For $L \subseteq ([k]^d)^{\omega}$, set
$$\nu_k(L)=\{ (\nu_k(w_1), \ldots \nu_k(w_d)): w_1, \ldots,w_d \in [k]^{\omega}, ((w_{1,i},\ldots ,w_{d,i}))_{i<\omega} \in L\}.$$

We can now formally define what it means for a subset of $[0,1]\subseteq \R$ to be $k$-automatic.
Let $k \in \N$ be greater than one, and let $d \in \N$ be greater than zero.

\begin{defn} \label{r-reg}
Say that $L\subseteq ([k]^d)^{\omega}$ is \emph{$k$-regular} if there is some B\"uchi automaton $\calA$ with alphabet $[k]^d$ that recognizes $L$.
Say that $A \subseteq [0,1]^d$ is \emph{$k$-automatic} if there is a B\"uchi automaton $\calA$ with alphabet $[k]^d$ that recognizes the maximal language $L \subseteq ([k]^d)^{\omega}$ such that $A=\nu_k(L)$.
Moreover, if this holds, say that $\calA$ \emph{recognizes} $A$.
\end{defn}

We also use the notation that for a B\"uchi automaton $\calA$ with alphabet $[k]$, $V_k(\calA)$ will denote the set of elements $x\in [0,1]$ for which some $k$-representation of $x$ is accepted by $\calA$.

The fact below follows immediately from the existence of a B\"uchi automaton with alphabet $\Sigma^2$ that accepts a pair of elements $x,y \in \Sigma^{\omega}$ precisely if both are $k$-representations of the same element of $[0,1]$.

\begin{fact}
For $A \subseteq [0,1]^d$, if there is some $k$-regular language $L \subseteq ([k]^d)^{\omega}$ such that $A=\nu_k(L)$, then the set of all $k$-representations of elements of $A$ is $k$-regular as well.
\end{fact}

Call an element $x \in [0,1]^d$ a \emph{$k$-rational} if there exists $w \in ([k]^d)^*$ such that $x=\nu_k(w\vec{0}^{\omega})$, where $\vec{0}$ is the $d$-tuple $(0, \dots, 0)$.
Clearly, these are the elements of $[0,1]^d$ whose coordinates can all be written as fractions with powers of $k$ in the denominators.

Throughout this paper $d$ denotes the (finite, but arbitrary) arity of the Euclidean space we are working in.
We use $\calA$ to denote both finite automata and B\"{u}chi automata, and we use $L$ to denote the subset of $([k]^d)^*$ that $\calA$ recognizes if it is a finite automaton, or to denote the subset of $([k]^d)^{\omega}$ that $\calA$ recognizes if it is a B\"{u}chi automaton.
If $\calA$ is a B\"{u}chi automaton, we will often use $A$ to denote $\nu_k(L)$, unless specified otherwise.
We will say that a B\"{u}chi automaton $\calA$ accepts $x \in [0,1]^d$ if $\calA$ accepts some $w \in ([k]^d)^{\omega}$ such that $\nu_k(w) = x$.

Given $\calA$ a trim B\"uchi automaton we let $\overline{A}$ denote the image under $\nu_k$ of the language that $\overline{\calA}$, the closure of $\calA$, recognizes.
In \cite{CLR15}, the authors show that every closed trim B\"uchi automaton recognizes a (topologically) closed set, hence the conflation of the set recognized by $\overline{\calA}$ with $\overline{A}$.
This conflation will be further justified in Section \ref{closed}.
In addition, we define closed $k$-automatic $\omega$-languages and the closure of a $k$-automatic $\omega$-language analogously.
If $L$ is a language (either a subset of $\Sigma^*$ or a subset of $\Sigma^{\omega}$) let $L^{pre} \subseteq \Sigma^*$ denote the set of all finite prefixes of elements of $L$.
Similarly, let $L_n^{pre}$ denote the set of all length-$n$ prefixes of elements of $L$, and let $L_{<n}^{pre}$ denote the set of all prefixes of $L$ with length at most $n$.

\subsection{Definition of entropy}

A key concept that turns out to be very helpful in the study of dimension of $k$-automatic sets is the notion of \textit{entropy.} 
The entropy of a formal language was perhaps first used for regular languages by Chomsky and Miller in  \cite{CM58} and was called entropy as an analogue for topological entropy by Hansen, Perrin, and Simon in \cite{HPS92}.
In \cite{AB11}, the authors note a seeming correspondence between the Hausdorff dimension of a $k$-automatic fractal and the entropy of the language of substrings of the base-$k$ expansions of its points.
Proving this conjecture is one of the main results of this paper. 
In order to do so, we find it most convenient to extend the definition of entropy to sequences of real numbers as follows:

\begin{defn}
    Let $(a_n)_{n\in\N}$ be a sequence of nonnegative real numbers with the property that infinitely many terms are nonzero. The \textit{entropy} of $a_n$ is defined as the limit superior:
    
    $$h((a_n)_n) = \limsup_{n \to \infty} \frac{\log a_n}{n}.$$
    
    The entropy $h(L)$ of an infinite language $L$ is the entropy of $(|L|_n)_n$.
\end{defn}

We choose to leave the entropy undefined for $a_n$ eventually zero, as this way the entropy is always a real number (and is nonnegative if $a_n$ is an integer sequence), which simplifies some results regarding entropy.

\subsection{Definition of box-counting dimension}

There are two different notions of dimension that will play large roles in this paper. 
The first is the concept of \textit{box-counting dimension,} also known as \textit{Minkowski dimension.} 
Intuitively, this is defined by quantifying how the number of boxes required to cover a given set increases as the size of the boxes decreases. 
This matches our intuition regarding ``nice'' sets that have a well-defined length, area, etc.
For instance, it is natural that to cover a polygonal area of $\R^2$ with boxes, when the boxes are half the size, this will require four times as many boxes. 
The box-counting dimension of such a polygon is $\frac{\log 4}{\log 2} = 2$.

In order to fully formalize this notion, many decisions must be made about the details. 
Is the ``size'' of a box its diameter or its side length? 
Must we use boxes, or could we use another shape, like a closed ball, instead? 
What if we allow the covering sets to be \textit{any} set of a given diameter? 
Should we place restrictions on the positioning of each box, such as requiring them to come from a grid? 
It turns out that most of these decisions have no effect on the resulting notion of dimension, i.e. they are equivalent.
Therefore, we use one of the several versions of the definition given in \cite{F03}:

\begin{defn}[\cite{F03}]
Let $X \subseteq \R^d$ be nonempty and bounded.
    \begin{enumerate}[(i)]
        \item We define $N(X, \epsilon)$ to be the number of sets of the form $I_{\vec{z}} = [z_1 \epsilon, (z_1 + 1) \epsilon] \times \dots \times [z_d \epsilon, (z_d + 1) \epsilon]$, where $\vec{z} = (z_1, \dots, z_d)$ are integers, required to cover $X$.
        \item The \textit{upper box-counting dimension} of $X$ is:
        
        $$d_{\overline{B}}(X) = \limsup_{\epsilon \to 0} \frac{\log N(X, \epsilon)}{\log \frac{1}{\epsilon}}.$$
        \item The \textit{lower box-counting dimension} of $X$ is:
        
        $$d_{\underline{B}}(X) = \liminf_{\epsilon \to 0} \frac{\log N(X, \epsilon)}{\log \frac{1}{\epsilon}}.$$
        \item If the upper and lower box-counting dimensions of $X$ are equal, we refer to their value as simply \textit{the box-counting dimension} $d_B(X)$.
    \end{enumerate}
\end{defn}

Box-counting dimension has several properties that justify its being called a dimension:

\begin{fact}[\cite{F03}]
\;
    \begin{enumerate}[(i)]
        \item If $X$ is a smooth $n$-manifold (embedded in $\R^d$), then $d_B(X) = n$.
        \item If $X \subseteq Y$, then $d_{\overline{B}}(X) \leq d_{\overline{B}}(Y)$ and $d_{\underline{B}}(X) \leq d_{\underline{B}}(Y)$.
        \item If $X = Y_1 \cup Y_2$, then $d_{\overline{B}}(X) = \max(d_{\overline{B}}(Y_1), d_{\overline{B}}(Y_2))$.
        \item Invertible affine transformations of $\R^d$ preserve $d_{\overline{B}}$ and $d_{\underline{B}}$.
    \end{enumerate}
\end{fact}

In addition to these, box-counting dimension has one more property that turns out to be quite useful (and that other notions of dimension do not possess):

\begin{fact}[\cite{F03}]
    Let $X \subseteq \R^d$ be nonempty and bounded. Then $d_{\overline{B}}(X) = d_{\overline{B}}(\overline{X})$, and $d_{\underline{B}}(X) = d_{\underline{B}}(\overline{X})$.
\end{fact}

\subsection{Definition of Hausdorff dimension}\label{Hdim}

Hausdorff dimension is the other notion of dimension that we will use in this paper. It is considerably more popular within fractal geometry, probably due to its compatibility with measure-theoretic notions. To define Hausdorff dimension, we must first define the notion of Hausdorff \textit{measure,} a family of outer measures on subsets of $\R^d$:

\begin{defn}[\cite{F03}]
    Let $X$ be a nonempty Borel subset of $\R^d$. For $s \geq 0, \epsilon > 0$ we define:
    
    $$\mu_H^s(X, \epsilon) = \inf \left\{\sum_{i=1}^\infty (\Diam U_i)^s : \{U_i\}_i \text{ is a collection sets of diameter at most $\epsilon$ covering $X$}\right\}$$
    
    The \textit{$s$-dimensional Hausdorff measure} of $X$, $\mu_H^s(X)$, is the limit of $\mu_H^s(X, \epsilon)$ as $\epsilon \to 0$.
\end{defn}

One precaution: recall that when we defined box-counting dimension above, we mentioned that it does not matter if the covering sets are boxes, balls, or any set with a given diameter. 
This is not the case with Hausdorff measure. 
Although the Hausdorff \textit{dimension} of $X$ would ultimately be the same if we changed these details in the above definition, the measure itself could be different.

Note that for subsets of $\mathbb{R}$ with Hausdorff measure one, the Hausdorff measure agrees with the Lebesgue measure.
A given set $X$ will only have a ``meaningful'' (i.e. nonzero and finite) Hausdorff measure for at most one value of $s$. 
Consider once more the example of a polygon in $\R^2$. The $2$-dimensional Hausdorff measure of a polygon is, up to a constant factor of $\frac{\pi}{4}$, its area. 
But the $s$-dimensional Hausdorff measure will be infinite for any $s < 2$ and zero for any $s > 2$. 
This suggests the following definition of Hausdorff dimension:

\begin{defn}[\cite{F03}]
    For $X \subseteq \R^d$ nonempty, the \textit{Hausdorff dimension} $d_H(X)$ is the unique real number $s$ such that $\mu_H^{s'}(X) = \infty$ for $s' < s$ and $\mu_H^{s'}(X) = 0$ for $s' > s$.
\end{defn}

Note that when $s' = s$, the Hausdorff measure may or may not be finite and may or may not be zero. What matters for determining dimension is the limiting behavior on either side of the critical value.

Hausdorff dimension has the properties we expect any notion of dimension to have:

\begin{fact}[\cite{F03}]
\;
    \begin{enumerate}[(i)]
        \item If $X$ is a smooth $n$-manifold (embedded in $\R^d$), then $d_H(X) = n$.
        \item If $X \subseteq Y$, then $d_H(X) \leq d_H(Y)$.
        \item If $X = Y_1 \cup Y_2$, then $d_H(X) = \max(d_H(Y_1), d_H(Y_2))$.
        \item Invertible affine transformations of $\R^d$ preserve $d_H$.
    \end{enumerate}
\end{fact}

In fact, Hausdorff dimension satisfies a stronger version of (iii) above:

\begin{fact}[\cite{F03}]
    If $X = \bigcup_{i \in \N} Y_i$, then $d_H(X) = \max_{i \in N} d_H(Y_i)$.
\end{fact}

As a corollary, the Hausdorff dimension of any countable set is zero (as that of a point is zero), so Hausdorff dimension is invariant under the addition or removal of a countable set of points. 
This is very unlike box-counting dimension: note that box-counting dimension is preserved under closures. 
Because $\R^d$ is separable, stability under closures and stability under the addition of countably many points are properties directly at odds with each other.

In particular, consider the set $X = \Q \cap [0, 1]$. 
The closure of $X$ is the interval $[0, 1]$, which has box-counting dimension $1$; hence $X$ has box-counting dimension $1$. 
Yet $X$ is countable and thus has Hausdorff dimension $0$. 
This gives an explicit example of when Hausdorff and box-counting dimension may differ. 
Note, however, that when they do differ, it is always the box-counting dimension that is higher:

\begin{fact}[\cite{F03}]
    Let $X$ be a nonempty and bounded subset of $\R^d$. Then $d_H(X) \leq d_{\underline{B}}(X) \leq d_{\overline{B}}(X)$.
\end{fact}

\section{Entropy and its relationship to dimension}\label{entropy}
\subsection{Properties of entropy}

It will be helpful to establish several properties of entropy before connecting it to fractal dimension.
First, we need a couple results from \cite{S85} concerning the monotonicity of entropy under the subset relation and union operation, and $\omega$-language prefixes.

\begin{fact}[\cite{S85}, Proposition 1]
\label{language_entropy_union}
    Let $L_1$ and $L_2$ be infinite languages.
    
    \begin{enumerate}[(i)]
        \item If $L_1 \subseteq L_2$, then $h(L_1) \leq h(L_2)$.
        \item $h(L_1 \cup L_2) = \max (h(L_1), h(L_2))$.
    \end{enumerate}
\end{fact}

\begin{fact}[\cite{S85}]\label{lang_entropy_union}
\label{language_entropy_summation}
    Let $L$ be an infinite language. Then:
    
    $$h(L^{pre}) = \limsup_{n \to \infty} \frac{\log |L|_{\leq n}}{n} = h(L).$$
\end{fact}

In the case where $L$ is closed under prefixes, we can define its entropy to be a limit, rather than a limit superior. 
To do this, we need a lemma concerning periodic functions:

\begin{lem}
\label{periodic_thing}
    Let $n > 1$, $P_1, \dots, P_n > 0$, and $\epsilon > 0$. 
    Let $S$ be the set of positive integers $z$ such that $\frac{z}{P_i} - \lfloor \frac{z}{P_i} \rfloor < \epsilon$ for $i = 1, \dots, n$. 
    Then there exists $N$ such that no $N$ consecutive positive integers lie outside $S$.
\end{lem}
\begin{proof}
    Note that with $k$ a positive integer, if $\frac{z}{kP} - \lfloor \frac{z}{kP} \rfloor < \frac{\epsilon}{k}$, then $\frac{z}{P} - \lfloor \frac{z}{P} \rfloor < \epsilon$. 
    By choosing least-common-multiple values of $P_i$ and smaller $\epsilon$, we can assume $P_i$ are all linearly independent over $\mathbb{Q}$. 
    Then at most one of them can be rational, say $P_1 = \frac{a}{b}$.
    
    We give $[0, 1)^{n-1}$ the toroidal metric where we identify opposite sides, i.e. the inherited metric from its canonical bijection with $\mathbb{R}^{n-1}/\mathbb{Z}^{n-1}$.
    The sequence $((\frac{ak}{P_2} - \lfloor \frac{ak}{P_2} \rfloor, \frac{ak}{P_3} - \lfloor \frac{ak}{P_3} \rfloor, \dots, \frac{ak}{P_n} - \lfloor \frac{ak}{P_n} \rfloor))_{k \in \mathbb{Z}^+}$ is dense in the unit box under the Euclidean metric; as the toroidal metric gives rise to a strictly coarser topology, the sequence is also dense in the toroidal metric.
    
    Now let $z$ be any positive integer, and let $Z$ be the smallest multiple of $a$ not less than $z$. 
    Let $B$ be the box in $[0, 1)^{n-1}$ containing the fractional parts of the elements of $[\frac{\epsilon}{3} - \frac{Z}{P_2}, \frac{2\epsilon}{3} - \frac{Z}{P_2}] \times \dots \times [\frac{\epsilon}{3} - \frac{Z}{P_n}, \frac{2\epsilon}{3} - \frac{Z}{P_n}]$. 
    Then by density, there is some $k \leq m$ such that $(\frac{ak}{P_2} - \lfloor \frac{ak}{P_2} \rfloor, \frac{ak}{P_3} - \lfloor \frac{ak}{P_3} \rfloor, \dots, \frac{ak}{P_n} - \lfloor \frac{ak}{P_n} \rfloor) \in B$. 
    Therefore, $(\frac{Z+ak}{P_2} - \lfloor \frac{Z+ak}{P_2} \rfloor, \frac{Z+ak}{P_3} - \lfloor \frac{Z+ak}{P_3} \rfloor, \dots, \frac{Z+ak}{P_n} - \lfloor \frac{Z+ak}{P_n} \rfloor) \in [\frac{\epsilon}{3}, \frac{2\epsilon}{3}]^{n-1}$.
    
    We have thus found a positive integer $Z+ak$ less than $z + a + am$ such that $\frac{Z+ak}{P_i} - \lfloor \frac{Z+ak}{P_i} \rfloor < \epsilon$ for $2 \leq i \leq n$; and $\frac{Z+ak}{P_1} - \lfloor \frac{Z+ak}{P_1} \rfloor = 0$ because $Z+ak$ is a multiple of $a$. 
    Setting $N = a + am + 1$ thus proves the lemma.
\end{proof}

The following fact is used in Corollary 9 of \cite{S85}, but not proven explicitly. 
Hence, we include the following result for completeness.

\begin{lem}
\label{language_entropy_prefix_limit}
	Let $L$ be a regular language closed under prefixes. 
	Then $\lim_{n \to \infty} \frac{\log |L|_n}{n}$ exists.
\end{lem}
\begin{proof}
	Let the deterministic finite automaton $\mathcal{A} = (Q, \Sigma, \delta, S, F)$ recognize $L$. 
	Let $\{q_1, \dots, q_m\} = Q$, and assume without loss of generality that $S = \{q_1\}$.
	Let $N_{n,i}$ be the number of strings of length $n$ of which there is a run in $\mathcal{A}$ from $q_1$ to $q_i$, and let $t_{i,j} = |\{\sigma \in \Sigma : q_j \in \delta(q_i, \sigma)\}|$. Then:
	
	$$N_{n+1,j} = \sum_{i \in I} t_{i,j} N_{n,i}$$
	$$|L|_{=(n+1)} = \sum_{i \in A} N_{n+1,i} = \sum_{j \in A} \sum_{i \in I} t_{i,j} N_{n,i}$$
	
	We thus have a system of recurrences that can be put in matrix form:
	
	\begin{align}
		\begin{bmatrix}
			N_{n+1, 1} \\
			N_{n+1, 2} \\
			\vdots \\
			N_{n+1, m} \\
			|L|_{n+1}
		\end{bmatrix}
		=
		\begin{bmatrix}
			t_{1,1} & t_{2,1} & \dots & t_{m,1} & 0 \\
			t_{1,2} & t_{2,2} & \dots & t_{m,2} & 0 \\
			\vdots & \vdots & \ddots & \vdots & \vdots \\
			t_{1,m} & t_{2,m} & \dots & t_{m,m} & 0 \\
			\sum_{j \in A} t_{1,j} & \sum_{j \in A} t_{2,j} & \dots & \sum_{j \in A} t_{m,j} & 0
		\end{bmatrix}
		\begin{bmatrix}
			N_{n, 1} \\
			N_{n, 2} \\
			\vdots \\
			N_{n, m} \\
			|L|_n
		\end{bmatrix}
	\end{align}
	
	Write the above as $\vec{v}_{n+1} = \mathbf{A} \vec{v}_n$. We have $\vec{v}_0 = [1, 0, \dots, 0, 1]^T$. 
	This then allows us to write $\vec{v}_{n} = \mathbf{A}^n \vec{v}_0$ and $|L|_n = [0, 0, \dots, 0, 1] \mathbf{A}^n \vec{v}_0$. 
	Using linear algebra techniques, we may then derive an expression for $|L|_n$:
	
	$$|L|_n = p_1(n) \lambda_1^n + \dots + p_r(n) \lambda_r^n$$
	
	where $p_i(n)$ are polynomials and $\lambda_i$ are the eigenvalues of $\mathbf{A}$. 
	Write $\lambda_i$ in polar form as $\rho_i e^{\mathbf{i} \theta_i}$. 
	Then taking the real part of both sides, we get:
	
	$$|L|_n = \mathrm{Re}(p_1(n)) (\rho_1^n \cos(n \theta_1)) + \dots + \mathrm{Re}(p_r(n)) (\rho_r^n \cos(n \theta_r))$$
	
	Let $n_1, n_2, \dots$ be an increasing sequence such that $\cos(n_i \theta_j) > 0.9$ for any $i, j$. 
	Let $n_{i+1} - n_i$ be bounded above by $N$; because the cosine function is periodic, we can find such an $N$ by Lemma \ref{periodic_thing}. Define:
	
	$$A_i = \mathrm{Re}(p_1(n_i)) \rho_1^{n_i} + \dots + \mathrm{Re}(p_r(n_i)) \rho_r^{n_i}$$
	
	Assume without loss of generality that $\rho_j$ above are in nonincreasing order, and let $k$ be such that $\rho_1 = \dots = \rho_k \neq \rho_{k+1}$. 
	The polynomial $\mathrm{Re}(p_1) + \dots + \mathrm{Re}(p_k)$ is eventually monotone and unbounded; if it is eventually decreasing, then $A_i \to -\infty$, which cannot be, as $|L|_{n_i}$ is nonnegative.
	
	Let $B_i = \mathrm{Re}(p_1(n_i)) \rho_1^{n_i} + \dots + \mathrm{Re}(p_k(n_i)) \rho_k^{n_i}$. 
	From the above, we know $B_i$ is eventually increasing to $+\infty$. Moreover, the ratio $\frac{A_i - B_i}{B_i}$ is vanishingly small. 
	So eventually, $0.8 B_i < |L|_{n_i} < 1.1 B_i$.
	
	Write $B_i = P(n_i) \rho_1^{n_i}$, with $P$ polynomial. Then:
	
	$$\lim_{i \to \infty} \frac{\log B_i}{n_i}$$
	$$= \lim_{i \to \infty} \frac{\log (P(n_i) \rho_1^{n_i})}{n_i}$$
	$$= \lim_{i \to \infty} \frac{\log P(n_i) + \log \rho_1^{n_i}}{n_i}$$
	$$= \lim_{i \to \infty} \frac{\log P(n_i)}{n_i} + \frac{n_i \log \rho_1}{n_i}$$
	$$= \log \rho_1.$$
	
	Therefore, because $0.8 B_i < |L|_{n_i} < 1.1 B_i$, we note that $\lim_{i \to \infty} \frac{\log |L|_{n_i}}{n_i} = \log \rho_1$ as well.
	
	Now, let $n_i \leq n'_i \leq n_{i+1}$. 
	Because $n_{i+1} \leq n'_i + N$, we have $|L|_{n_{i+1}} \leq |L|_{n'_i} |\Sigma|^N$. 
	By the same logic, $|L|_{n'_i} \leq |L|_{n_i} |\Sigma|^N$. 
	Letting $M = |\Sigma|^N$ then gives us $\frac{1}{M} |L|_{n_{i+1}} \leq |L|_{n'_i} \leq M |L|_{n_i}$. So:
	
	$$\log \rho_1 = \lim_{i \to \infty} \frac{\log |L|_{n_{i+1}}}{n_{i+1}}$$
	
	Adding a bounded value to the top or bottom of this fraction will not affect the limit:
	
	$$= \lim_{i \to \infty} \frac{\log \frac{1}{M} + \log |L|_{n_{i+1}}}{n_{i+1} + (n'_i - n_{i+1})}$$
	$$= \lim_{i \to \infty} \frac{\log \left( \frac{1}{M} |L|_{n_{i+1}} \right)}{n'_i}$$
	$$\leq \lim_{i \to \infty} \frac{\log |L|_{n'_i}}{n'_i}$$
	$$\leq \lim_{i \to \infty} \frac{\log \left( M |L|_{n_i} \right)}{n'_i}$$
	$$= \lim_{i \to \infty} \frac{\log M + \log |L|_{n_i}}{n_i + (n'_i - n_i)}$$
	$$= \lim_{i \to \infty} \frac{\log |L|_{n_i}}{n_i} = \log \rho_1.$$
	
	Therefore, $\lim_{i \to \infty} \frac{\log |L|_{n'_i}}{n'_i} = \log \rho_1$. 
	This is independent of the choice of each $n'_i$ following the constraints, so we can cover the sequence $\frac{\log |L|_{n}}{n}$ by subsequences converging to $\log \rho_1$. Therefore, the entire sequence must also converge to this limit.
\end{proof}

\begin{cor}
\label{language_entropy_prefix_limit_summation}
    Let $L$ be a regular language closed under prefixes. Then $\lim_{n \to \infty} \frac{\log |L|_{\leq n}}{n}$ exists.
\end{cor}
\begin{proof}
    With $\Sigma$ the alphabet of $L$, assume without loss of generality that $\$ \notin \Sigma$. Let $\Sigma' = \Sigma \cup \{\$\}$ and $L' = L\$^*$. 
    Then $L'$ is prefix-closed, and there is a one-to-one correspondence between strings of length at most $n$ in $L$ and strings of length exactly $n$ in $L'$ (realized by adding the appropriate number of $\$$). 
    The result then immediately follows from Lemma \ref{language_entropy_prefix_limit}.
\end{proof}

\subsection{Entropy and box-counting dimension}

Note that box-counting dimension appears to be quite hard to compute by the definition we have given, because there are fairly general quantifications over multiple variables. 
But with a bit of work, the process of proving the box-counting dimension of a set can be simplified. 
In particular, we have the following result showing that we do not need to check many values of $\epsilon$ and $h$:
	
\begin{lem}
\label{box_counting_simpler_limit}
	Let $X \subseteq [0, 1]^d$. Let $r > 1$, and assume that the limit:
	
	$$L = \lim_{n \to \infty} \frac{\log N(X, r^{-n}, \vec{0})}{\log r^n}$$
	
	exists. Then $d_B(X) = L$.
\end{lem}
\begin{proof}
	Let $(\epsilon_i)_i$ be any sequence of positive values converging to zero, and let $(\vec{h}_i)_i$ be any sequence of values in $[0, 1)^d$. We will show that:
	
	$$\lim_{i \to \infty} \frac{\log N(X, \epsilon_i, \vec{h}_i)}{\log \frac{1}{\epsilon_i}} = L.$$
	
	Choose some arbitrary $i$, and find $n_i$ such that $r^{-n_i} > \epsilon_i > r^{-n_i - 1}$. 
	Note first that $n_i \to \infty$ as $i \to \infty$; otherwise, there would exist $n^*$ such that infinitely many $n_i < n^*$, and hence infinitely many $\epsilon_i > r^{-n^* - 1}$, which cannot be.
	
	Consider the boxes of the form $I_{\vec{z}} = [z_1 r^{-n_i}, (z_1 + 1) r^{-n_i}] \times \dots \times [z_d r^{-n_i}, (z_d + 1) r^{-n_i}]$ where $\vec{z} = (z_1, \dots, z_d)$ are integers. 
	Similarly, choose an arbitrary vector of integers $\vec{z}' = (z'_1, \dots, z'_d)$ and let $I' = [(z'_1 - h_{i,1}) \epsilon_i, (z'_1 - h_{i,1} + 1) \epsilon_i] \times \dots \times [(z'_d - h_{i,d}) \epsilon_i, (z'_d - h_{i,d} + 1) \epsilon_i]$. 
	Note that because the $I_{\vec{z}}$ are adjacent (i.e. they partition $\mathbb{R}^d$) and longer than $I'$, then there exist at most $2^d$ such $I_{\vec{z}}$ that cover $I'$. 
	We chose $I'$ arbitrarily, so therefore, $N(S, r^{-n_i}, 0) \leq 2^d N(S, \epsilon_i, \vec{h}_i)$.
	
	Similarly, choose an arbitrary $\vec{z}$ and let $I = [z_1 r^{-n_i - 1}, (z_1 + 1) r^{-n_i - 1}] \times \dots \times [z_d r^{-n_i - 1}, (z_d + 1) r^{-n_i - 1}]$, and consider the intervals of the form $I_{\vec{z}'} = [(z'_1 - h_{i,1}) \epsilon_i, (z'_1 - h_{i,1} + 1) \epsilon_i] \times \dots \times [(z'_d - h_{i,d}) \epsilon_i, (z'_d - h_{i,d} + 1) \epsilon_i]$ for integer vectors $\vec{z}'$. 
	Note that because the $I_{\vec{z}'}$ are adjacent (i.e. they partition $\mathbb{R}$) and longer than $I$, then there exist at most $2^d$ such $I_{\vec{z}'}$ that cover $I$. 
	We chose $I$ arbitrarily, so therefore, $N(X, \epsilon_i, \vec{h}_i) \leq 2^d N(X, r^{-n_i-1}, 0)$.
	
	These inequalities give us:
	
	$$\frac{\log \frac{1}{2^d} N(X, r^{-n_i}, 0)}{\log r^{n_i + 1}} \leq \frac{\log N(X, \epsilon_i, \vec{h}_i)}{\log \frac{1}{\epsilon_i}} \leq \frac{\log 2^d N(X, r^{-n_i-1}, 0)}{\log r^{n_i}}.$$
	
	Apply laws of logarithms:
	
	$$\frac{- \log 2^d + \log N(X, r^{-n_i}, 0)}{\log r + \log r^{n_i}} \leq \frac{\log N(X, \epsilon_i, \vec{h}_i)}{\log \frac{1}{\epsilon_i}} \leq \frac{\log 2^d + \log N(X, r^{-n_i-1}, 0)}{-\log r + \log r^{n_i + 1}}.$$
	
	Now note that as long as $a_i, b_i \to \infty$ with $C_1, C_2$ constant, we have $\lim_{i \to \infty} \frac{C_1 + a_i}{C_2 + b_i} = \lim_{i \to \infty} \frac{a_i}{b_i}$. So:
	
	$$\lim_{i \to \infty} \frac{- \log 2^d + \log N(X, r^{-n_i}, 0)}{\log r + \log r^{n_i}} = \lim_{i \to \infty} \frac{\log N(X, r^{-n_i}, 0)}{\log r^{n_i}} = L.$$
	
	Similarly:
	
	$$\lim_{i \to \infty} \frac{\log 2^d + \log N(X, r^{-n_i-1}, 0)}{-\log r + \log r^{n_i + 1}} = \lim_{i \to \infty} \frac{\log N(X, r^{-n_i-1}, 0)}{\log r^{n_i + 1}} = L.$$
	
	Since $\frac{\log N(X, \epsilon_i, \vec{h}_i)}{\log \frac{1}{\epsilon_i}}$ is bounded between two sequences converging to $L$, we conclude that its limit is $L$ as well.
\end{proof}

\begin{lem}
\label{box_counting_entropy}
    Let $L$ be an $\omega$-language of base-$k$ representations of points in $[0, 1]^d$. 
    Let $X = \nu_k(L)$. Then $d_B(X) = \frac{1}{\log k} h(L^{pre})$.
\end{lem}
\begin{proof}
    Lemma \ref{language_entropy_prefix_limit} tells us that the entropy of $L^{pre}$ is defined as a limit and not just as a limit superior. 
    By Lemma \ref{box_counting_simpler_limit}, it then suffices to show:
    
    $$\lim_{n \to \infty} \frac{\log N(X, k^{-n}, \vec{0})}{\log k^n} = \frac{1}{\log k} \lim_{n \to \infty} \frac{\log |L^{pre}|_n}{n}.$$
    
    Now consider each string $w$ in $L^{pre}$ of length $n$. 
    The infinite strings with this prefix represent precisely the numbers in $I_{\vec{z}} = [z_1 k^{-n}, (z_1 + 1) k^{-n}] \times \dots \times [z_d k^{-n}, (z_d + 1) k^{-n}]$ where $\vec{z} = (z_1, \dots, z_d)$ is the integer vector with base-$k$ representation given by $w$. 
    So all of these boxes for all such strings $w$ will cover $X$; thus, $N(X, k^{-n}, \vec{0}) \leq |L^{pre}|_n$.
    
    This covering may not be optimal, because the boxes $I_{\vec{z}}$ are not disjoint; they overlap at points with at least one coordinate that is $k$-rational. 
    So a single $I_{\vec{z}}$ may contain points that the above covering method would place in any adjacent box. 
    However, this is the only overlap, so a box in the optimal covering corresponds to at most $3^d$ boxes in the above covering. 
    Hence $N(X, k^{-n}, \vec{0}) \geq \frac{1}{3^d} |L^{pre}|_n$. 
    Therefore:
    
    $$\frac{1}{\log k} \lim_{n \to \infty} \frac{\log |L^{pre}|_n}{n}$$
    $$= \frac{1}{\log k} \lim_{n \to \infty} \frac{\log \frac{1}{3^d} + \log |L^{pre}|_n}{n}$$
    $$= \frac{1}{\log k} \lim_{n \to \infty} \frac{\log \left(\frac{1}{3^d} |L^{pre}|_n\right)}{n}$$
    $$\leq \frac{1}{\log k} \lim_{n \to \infty} \frac{\log N(X, k^{-n}, \vec{0})}{n}$$
    $$\leq \frac{1}{\log k} \lim_{n \to \infty} \frac{\log |L^{pre}|_n}{n}.$$
    
    It follows that the two limits are equal, as required.
\end{proof}

We conclude this section by noting that the choice of regular language to associate with a $k$-automatic set is not always obvious. 
In particular, in their conjecture regarding the connection between entropy and dimension, Adamczewski and Bell \cite{AB11} use the regular language of \textit{substrings} of the base-$k$ expansions, not the language of prefixes as we are using here. 
We will show that this does not make a difference.

This result follows from \cite{S85}, in which the author shows that the entropy of a $k$-regular $\omega$-language $L$ is equal to that of the (normal) regular language $L^{pre}$.
Theorem 14 of \cite{S85} further shows that if $L$ is an $\omega$-language then there is a strongly-connected $\omega$-language $L'$ and a word $w \in \Sigma$ such that $\{v: v=w\ell, \ell \in L'\} \subseteq L$, and $h(L)=h(L')$.
It follows that the entropy of $L$ is equal to the maximum over the entropies of the subsets corresponding to the strongly connected components of any B\"uchi automaton $\calA$ recognizing $L$, which together implies the entropy of $L$ is the same as that of the set of substrings of $L$.

\begin{fact}[\cite{S85}]
\label{entropy_pre_sub}
    Let $L$ be a regular language closed under prefixes, and let $L^{sub}$ be the regular language of substrings of strings in $L$. Then $h(L) = h(L^{sub})$.
\end{fact}

\section{The closed case}\label{closed}

\subsection{Spectral radius and box-counting dimension}

The goal of this subsection will be to produce a method for computing the box-counting dimension of a closed $k$-automatic set based on its B\"uchi automaton. 
We will do this by first computing the entropy of the corresponding language of prefixes and then applying Lemma \ref{box_counting_entropy}.

\begin{defn}\label{adjmat}
Let $\mathcal{A}$ be a B\"uchi automaton. 
Suppose that $Q$, the set of states of $\calA$, is size $n$ and let $\iota:\{1,\ldots,n\} \to Q$ be a fixed bijection.
Then we associate to $\calA$ a weighted adjacency matrix $\mathbf{M}(\calA,s)=(m_{i,j})$ given by:
\[
m_{i,j} = \left (\frac{|\{\sigma \in \Sigma: (\iota(i),\iota(j),\sigma) \in \Delta\}|}{ k}\right)^{s}.
\]
Recall that $\Sigma$ is the alphabet and $\Delta$ is the transition function of $\calA$.
\end{defn}

We will need the following fact regarding sequences defined by repeated matrix multiplication. 
The following result appears in \cite{S85}:

\begin{fact}
\label{language_radius_entropy}
    Let $L$ be a regular language closed under prefixes. Let $\mathcal{A}$ be a finite automaton recognizing $L$, and assume that $\mathcal{A}$ is deterministic, trim, and has every state accepting. 
    Label the states of $\mathcal{A}$ by numbers $1$ through $m$.
    Then the entropy of $L$ is equal to the spectral radius $\sprad{k\mathbf{M}(\calA,1)}$ of weighted adjacency matrix $\mathbf{M}(\calA,1)$ scaled by the constant $k$.
\end{fact}

We know a lot about the dimension of $V_k(\mathcal{A})$ when it is a closed set, and we know that if $\mathcal{A}$ is deterministic (and trim), then we can take the closure of $V_k(\mathcal{A})$ by making all states accepting. 
It is not immediately obvious that this will hold when $\mathcal{A}$ is nondeterministic, but in the following lemma we show this is case.
The following result is essentially a corollary of Lemma 58 in \cite{CLR15}:

\begin{lem}
\label{automaton_closure}
    Let $\mathcal{A}$ be a trim B\"uchi automaton. Then $V_k(\overline{\calA}) = \overline{V_k(\mathcal{A})}$.
\end{lem}
\begin{proof}
    First, we prove that $V_k(\overline{\calA})$ is closed. 
    Let $x \in \overline{V_k(\overline{\calA})}$. 
    Then there exist $(x_m)_{m \in \N} \in V_k(\overline{\calA})$ with $x_m \to x$. 
    Without loss of generality assume that either $x_m < x$ for all $m$, or else $x_m > x$ for all $m$. 
    We will assume the latter; the proof in the former case is analogous.
    
    Let $w$ be the infinite base-$k$ representation of $x$ such that if $x$ is $k$-rational, then $w$ ends in $0^\omega$; this uniquely identifies $w$. 
    (We choose this representation because we assumed $x_m > x$.) 
    Then for all $n$ there exists $m_n$ such that the first $n$ characters of some base-$k$ representation of $x_{m_n}$ are the first $n$ characters of $w$. Let $w_n$ be this base-$k$ representation for a given $n$.
    
    Now, we will define a graph $G$ inductively. 
    The vertex set will be a subset of $\Sigma^* \times Q^*$. 
    Our initial condition is that $G$ contains the vertices $(\epsilon, q)$, where $q$ ranges over $S$ (the set of start states of $\overline{\calA}$). Then for every vertex $(u, v)$, we require that $G$ contain the vertex $(uc, vq)$ where $uc$ is a prefix of $w$ and $\overline{\calA}$ transitions from the last state in $v$ to state $q$ on $c$; and we require that $G$ contain the edge from $(u, v)$ to $(uc, vq)$.
    
    Observe that:
    
    \begin{itemize}
        \item $G$ has finitely many connected components: as we have defined $G$, every vertex is connected to a vertex whose string is shorter. 
        So if the length of the string for a vertex is $n$, that vertex has a path of length $n$ to an initial vertex.
        \item $G$ is locally finite: every vertex has at most $k|Q|$ incident edges.
        \item $G$ is infinite: if $w_n$ is accepted, there must be an accepting path for the first $n$ characters of $w_n$, which are also the first $n$ characters of $w$. 
        This gives a vertex for every natural number $n$.
    \end{itemize}
    
    Therefore, we can apply K\H{o}nig's lemma to one of the connected components of $G$ and conclude that $G$ has an infinite path. 
    Moreover, note that because strings have finite length, infinitely many of the edges in the path must lead to a longer string; and because there is only one edge from any vertex to a shorter string, it in fact must be the case that eventually the path only contains edges to longer strings. 
    Choose a vertex sufficiently far into the path that this is the case; there is a path from an initial vertex to this vertex and then infinitely farther through longer and longer strings; and this path must correspond to an accepting path for $w$. 
    So $\overline{\calA}$ accepts $w$, and thus, $V_k(\overline{\calA})$ is closed.
    
    Now it remains to show that $V_k(\overline{\calA})$ is the closure of $V_k(\mathcal{A})$.
    Note that any accepting run in $\mathcal{A}$ is also accepting in $\overline{\calA}$, so $V_k(\mathcal{A}) \subseteq V_k(\overline{\calA})$. 
    Moreover, let $w$ be accepted by $\overline{\calA}$. 
    Then for any $n$, the first $n$ characters of $w$ have a run to some state in $\overline{\calA}$. 
    As the only difference between $\mathcal{A}$ and $\overline{\calA}$ is the set of accept states, there is also a run for these characters in $\mathcal{A}$, and because $\mathcal{A}$ is trim, this run can be extended to an accepting run. 
    So every infinite string accepted by $\overline{\calA}$ has, for all $n$, its first $n$ characters in common with some infinite string accepted by $\mathcal{A}$. 
    It follows that $V_k(\overline{\calA}) \subseteq \overline{V_k(\mathcal{A})}$. 
    But the only closed subset of $\overline{V_k(\mathcal{A})}$ that is a superset of $V_k(\mathcal{A})$ is $\overline{V_k(\mathcal{A})}$ itself.
\end{proof}

\subsection{Spectral radius and Hausdorff dimension}

In \cite{MW88}, the objects whose Hausdorff dimensions they describe are called ``geometric graph directed constructions.''
These constructions would be described in the current terminology of metric geometry as GDIFSs that satisfy the open set condition and have compact attractors.

Recall the open set condition for GDIFSs as defined in \cite{E08}.

\begin{defn}
If $\mathcal{G} = (V,E,s,t,X,S)$ is a GDIFS, then it satisfies the \emph{open set condition} precisely if for all $v \in V$ there are nonempty open sets $U_v \subseteq \mathbb{R}^d$ such that
$$U_u \supseteq f_e[U_v]$$
for all $u,v \in V$ such that $e \in E_{u,v}$, and additionally,
$$f_e[U_v] \cap f_{e'}[U_{v'}] = \emptyset$$
for all $u, v, v' \in V$ such that $e \in E_{u,v}$ and $e'\in E_{u,v'}$.
\end{defn}

\begin{thm}
Suppose that $X \subseteq [0,1]^d$ is a $k$-automatic set recognized by a deterministic B\"uchi automaton $\mathcal{A}$ with corresponding GDIFS $\mathcal{G}_{\mathcal{A}} = (V,E,s,t,X,S)$.
Then $\mathcal{G}_{\mathcal{A}}$ satisfies the open set condition as defined in \cite{E08}.
\end{thm}
\begin{proof}
Let $X$, $\mathcal{A}$ and $\mathcal{G}_{\mathcal{A}}$ be as in the hypotheses.
To see that the open set condition holds, for each $v \in V$ let $U_v = (0,1)^d$.
Observe that for any $e \in E$, it is the case that $f_e[U_v] = \{\frac{x+\sigma_e}{r}: x\in U_v\}$, where $\sigma_e \in \Sigma$ is the label of the transition arrow in $\mathcal{A}$ that corresponds to the edge $e$.
Hence it is clear that $U_u \supseteq f_e(U_v)$ for any vertices $u,v \in V$ such that $e$ is an edge from $u$ to $v$, since
$$( 0,1 ) ^d \supseteq \prod_{i=1}^d\left(\frac{\sigma_{e,i}}{r},\frac{ \sigma_{e,i}+1}{r}\right).$$

Suppose now that $u,v,v' \in V$ and that $e \in E_{u,v}$ and $e' \in E_{u,v'}$.
Due to $\mathcal{A}$ being deterministic, it cannot be the case that the transition arrows corresponding to $e$ and $e'$ respectively in $\mathcal{A}$ have the same label. 
Thus we know $\sigma_e \neq \sigma_{e'}$, and as a consequence if an element $\vec{a}$ is in $\{\frac{x+\sigma_e}{r}: x\in (0,1)^d\}$, then it cannot be the case that $\vec{a} \in \{\frac{x+\sigma_{e'}}{r}: x\in U_v\}$, since such sets carve out the interior of disjoint subcubes of $(0,1)^d$ of size $\frac{1}{r}^d$.
Hence the open set condition is satisfied.

\end{proof}

The following corollary is immediate from the fact that a closed B\"uchi automaton is weak, by applying the above theorem.
Recall that weak automata (which include all closed automata) have an equivalent deterministic B\"uchi automaton \cite{PP04}.

\begin{cor}\label{open}
All closed B\"uchi automata satisfy the open set condition.
\end{cor}

In order to apply Theorem 5 from the paper \cite{MW88} of Mauldin and Williams, we need the following lemma concerning the equality of the Hausdorff dimension for two strongly connected automata which only differ in which states are initial.

\begin{lem}\label{strongly_connected_start_state}
    Let $\mathcal{A}$ and $\mathcal{A}'$ be two strongly connected B\"uchi automata with the same set of states, set of accept states, and transition relation. Then $d_H(V_k(\calA))$ = $d_H(V_k(\calA'))$.
\end{lem}
\begin{proof}
    Note that $V_k(\calA) = \bigcup_{q_i \in S} V_k(\calA_i)$, where $S$ is the set of start states of $\calA$ and $\calA_i$ is a modification of $\calA$ where $q_i$ is the only start state. 
    Therefore $d_H(V_k(\calA)) = \max_{q_i \in S} d_H(V_k(\calA_i))$. 
    Let $r$ be the value of $i$ giving maximal dimension, i.e. $d_H(V_k(\calA)) = d_H(V_k(\calA_r))$. 
    Define $\calA'_s$ and choose $s$ analogously, such that $d_H(V_k(\calA')) = d_H(V_k(\calA'_s))$ and such that $q_s$ is the only start state in $\calA'_s$.
    
    Because $\calA_r$ is strongly connected, it must contain a path $P$ from $q_s$ to $q_r$; let $v \in ([k]^d)^*$ be a string witnessing path $P$. 
    Let $w \in ([k]^d)^\omega$ be accepted by $\calA_r$. 
    Then $P$ concatenated with an accepting run of $w$ in $\calA_r$ forms an accepting run of $vw$ in $\calA'_s$. 
    Moreover, note that:
    
    $$\nu_k(vw) = \sum_{i=0}^\infty \frac{(vw)_i}{k^{i+1}}$$
    $$= \sum_{i=0}^{|v|-1} \frac{(vw)_i}{k^{i+1}} + \sum_{i=|v|}^\infty \frac{(vw)_i}{k^{i+1}}$$
    $$= \sum_{i=0}^{|v|-1} \frac{v_i}{k^{i+1}} + \sum_{i=0}^\infty \frac{w_i}{k^{|v|+i+1}}$$
    $$= \nu_k(v\vec{0}^\omega) + k^{-|v|} \nu_k(w).$$
    
    Let $f : [0, 1]^d \to [0, 1]^d$ map $x$ to $\nu_k(v\vec{0}^\omega) + k^{-|v|} x$; then the above chain of equalities gives us that $f(V_k(\calA_r)) \subseteq V_k(\calA'_s)$. 
    Because $f$ is an invertible affine transformation, it follows that $d_H(V_k(\calA_r)) \leq d_H(V_k(\calA'_s))$. 
    The analogous argument in the opposite direction gives us $d_H(V_k(\calA'_s)) \leq d_H(V_k(\calA_r))$. 
    Thus:
    
    $$d_H(V_k(\calA)) = d_H(V_k(\calA_r)) = d_H(V_k(\calA'_s)) = d_H(V_k(\calA')).$$
\end{proof}

In \cite{MW88} Mauldin and Williams work with what they call a ``geometric graph directed construction'' on $\R^m$. 
We observe that the GDIFS associated to a closed B\"uchi automaton is nearly an instance of such a ``geometric graph directed construction.''
In \cite{MW88}, however, the authors consider only iterated function systems directed by a graph, whereas the GDIFSs associated to closed automata are in general multigraphs. 
To account for this, though, we observe that we can ``translate'' between multigraph representations of B\"uchi automata and automaton diagrams which are true digraphs.

\begin{defn}\label{dg}
For a deterministic B\"uchi automaton $\calA=(Q,\Sigma,\delta,S,F)$, consider the automaton $(Q \times \Sigma, \Sigma, \delta', S', F \times \Sigma)$, with $\delta'$ and $S'$ defined as follows:
\begin{itemize}
    \item For all $\sigma, \tau \in \Sigma$ and $q, s \in Q$, if $\delta(q, \sigma) = \varnothing$, then $\delta'((q, \tau), \sigma) = \varnothing$. 
    If $\delta(q, \sigma) = \{s\}$, then $\delta'((q, \tau), \sigma) = \{(s, \sigma)\}$. 
    (These are the only possibilities, because $\calA$ is deterministic.)
    \item Fix some $\sigma_0 \in \Sigma$ arbitrarily; then with $S = \{q\}$, we have $S' = \{(q, \sigma_0)\}$.
\end{itemize}
Since this automaton is not necessarily trim, let $\calA_{dg}$ denote the automaton that results from removing any states that are not accessible from the start state or not co-accessible from some accept state.
\end{defn}

\begin{lem}\label{multigraph_to_digraph}
If the automaton $\calA$ is a closed B\"uchi automaton, then $\calA_{dg}$ is an equivalent B\"uchi automaton that is also closed.
\begin{proof}
To see that $\calA$ and $\calA_{dg}$ are equivalent as automata, observe that if $(q_i)_{i\in \N}$ is an accepting run of $(\tau_i)_{i\in \N}$ in $\calA$, then set $\tau_{-1}=\sigma_0$ and observe that the run $(q_i, \tau_{i-1})_{i \in \N}$ of $(\tau_i)_{i\in \N}$ in $\calA_{dg}$ is an acceptance run for $\calA_{dg}$.
Similarly if $(q_i, \sigma_{i})_{i \in \N}$ is an acceptance run of $(\tau_i)_{i\in \N}$ in $\calA_{dg}$, then $(q_i)_{i\in \N}$ is an acceptance run in $\calA$ for the same string.

Since we define $\calA_{dg}$ to be trim and since $Q=F$ for $\calA$, we observe $Q\times \Sigma = F\times \Sigma$ makes $\calA_{dg}$ closed precisely if $\calA$ is closed.
\end{proof}
\end{lem}

The next corollary follows from the remarks made in the proof of Lemma \ref{multigraph_to_digraph} and essentially says that connected components are preserved by the construction in Definition \ref{dg}.

\begin{cor}\label{scdg}
If $C$ is a strongly connected component of $\calA$, then there is a corresponding strongly connected component of $\calA_{dg}$, call it $C_{dg}$, such that the induced sub-automata generated by $C$ and $C_{dg}$ (respectively) recognize the same strings up to prefixes. 
In other words, let $L$ be recognized by the induced sub-automaton generated by $C$, and likewise for $L_{dg}$ and $C_{dg}$; then there exist words $u, v$ such that $L = uL_{dg}$ and $L_{dg} = vL$.
\end{cor}
\begin{proof}
    First, we note that the choice of start state in each induced sub-automaton does not matter. 
    In a strongly connected automaton, if the start state is moved, the resulting language and original language are equal up to prefixes (as defined in the statement of the corollary). So we will allow the start states to be chosen freely.
    
    Let $\calA_C$ denote the sub-automaton of $\calA$ consisting of the states and transitions contained within $C$, with arbitrarily chosen $q \in C$ as its only start state. 
    Then by strong connectedness, there exists a word $w$ and a corresponding run $(q_i)_{0 \leq i \leq |w|}$ where $q_0 = q_n = q$ and where, for each transition within $C$, there exists an $i$ witnessing the transition, i.e. if there is an arrow from state $r$ to state $s$ on the character $\sigma$, then there exists $i$ such that $q_{i-1} = r, q_i = s, w_i = \sigma$.
    
    Let $C_{dg}$ be the strongly connected component of $(q, w_n)$ in $\calA_{dg}$, and let $\calA_{C_{dg}}$ be its induced sub-automaton with $(q, w_n)$ as the start state. 
    Then there is a corresponding path $((q_i, w_i))_{0 \leq i \leq |w|}$ (letting $w_0 = w_n$) in $\calA_{C_{dg}}$; this path is a cycle, so each state $(q_i, w_i)$ is in $C_{dg}$.
    
    Let $u$ be any infinite word accepted by $\calA_C$. 
    Following the proof of the previous lemma, note that because each transition in the acceptance run for $u$ is also somewhere in the run for $w$, and therefore it has a corresponding transition in $\calA_{C_{dg}}$. 
    So $u$ is accepted by $\calA_{C_{dg}}$. 
    The reverse inclusion follows immediately from the same argument as in Lemma \ref{multigraph_to_digraph}.
\end{proof}

The following is implied by the results in \cite{S89} in which the author shows that for closed languages the Hausdorff dimension and entropy agree, in conjunction with the countable additivity of entropy, as shown in \cite{S85}.

\begin{lem}\label{hdim}
Let $\calA$ be a closed B\"uchi automaton.
Let $SC(\calA)$ denote the set of strongly connected components of $\calA$.
For each strongly connected component $C$ of $\calA$, let $\calA_{C}$ denote the sub-automaton of $\calA$ consisting of the states and transitions contained within $C$, with arbitrarily chosen $q \in C$ as its only start state.

Then 
\[
d_H(V_k(\calA)) = \max_{C \in SC(\calA)}\left( d_H( V_{k}(\calA_{C})) \right).
\]
\begin{proof}
By Lemma \ref{strongly_connected_start_state}, the Hausdorff dimension of the automaton $\calA_C$ does not depend on the choice of $q$, its start state.
Let $K$ be the attractor of $G_{\calA}$, i.e. the set that $\calA$ recognizes.
We recall the statement of Theorem 5 of \cite{MW88}.
It states that if $K= \bigcup_{v \in V} K_v$ is the attractor of GDIFS $G$, then $d_H(K_v) = \alpha_v = \max \{ \alpha_H| H \in SC_v(G)\}$.
Here, $SC_v(G)$ denotes the set of strongly connected components of $G$ that are accessible from $v$.

By Corollary \ref{open} we know that $\calA_{dg}$ satisfies the definition in \cite{MW88} of a geometric graph directed construction.
By Lemma \ref{multigraph_to_digraph} we know that the automaton $\calA_{dg}$ recognizes the same set as $\calA$, hence $d_H(V_k(\calA))=d_H(\calA_{dg})$.
For each strongly connected component $C$ of $\calA$, 
let $\calA_C$ be the automaton whose automaton diagram is given by taking $C$ and assigning an arbitrary start state from those in the vertex set of $C$.
By Corollary \ref{scdg}, for each $C$ a strongly connected component of $G$ there is a corresponding strongly connected component $C_{dg}$ of $\calA_{dg}$ such that $d_H(V_k(\calA_C))=d_H(V_k(\calA_{C,dg}))$, where $\calA_{C,dg}$ is an automaton whose automaton diagram is given by taking $C_{dg}$ and assigning an arbitrary start state from those in the vertex set of $C_{dg}$.

Since $\calA_{dg}$ has a true digraph structure (rather than that of a multigraph), we can apply Theorem 3 of \cite{MW88} to obtain that the Hausdorff dimension of $V_k(\calA_{C,dg})$ is the unique $\alpha$ such that $\sprad{\mathbf{M}(\calA_C, \alpha)} =1$.
Note that the value of $\max_{C \in SC(\calA)} d_H( V_{k}(\calA_{C}))$ is well-defined by Lemma \ref{strongly_connected_start_state}.
By Theorem 4 of \cite{MW88}, we conclude that the dimension of $V_k(\calA_{dg})$ is the maximum over all such $\alpha$ for $\calA_{C,dg} \in SC(\calA_{dg})$.

\end{proof}
\end{lem}

\begin{thm}\label{closed_equality}
If $X$ is a closed $k$-automatic set in $\R$ recognized by closed automaton $\calA$, then $d_H(X) = d_B (X)= \frac{1}{\log (k)} h(L(\calA))$, where $L(\calA)$ is the set of strings $\calA$ recognizes.
\begin{proof}
By Lemma \ref{hdim}, it suffices to show that $d_H(X) = d_B (X)= \frac{1}{\log (k)} h(L(\calA_{dg}))$, since $L(\calA) = L(\calA_{dg})$ by Lemma \ref{multigraph_to_digraph}.
Suppose that $\calB$ is the induced sub-automaton given by taking a strongly connected component of $\calA_{dg}$ and selecting a start state arbitrarily.
Let $L$ be the language that $\calB$ recognizes as a subset of $[k]^{\omega}$.
By Theorem 3 of \cite{MW88} the Hausdorff dimension of $V_k(\calB)$ is the unique $\alpha$ such that $\sprad{M(\calB, \alpha)} =1$.

By computations from \cite{MW88}, we know that if $\vec{u} = (1, \ldots ,1) \in \R^m$, where $\calB$ has $m$ states, then
$$\sprad{\mathbf{M}(\calB, \alpha)}= \lim_{n \to \infty} \Vert \mathbf{M}(\calB, \alpha)^n\vec{u} \Vert^{1/n}_1.$$
Moreover, the authors show by induction in \cite{MW88} that the following holds:
$$\Vert \mathbf{M}(\calB, \alpha)^n\vec{u} \Vert_1 = \sum_{w\in L^{pre}_n} \left( \frac{1}{k^n}\right)^{\alpha}$$
where $L^{pre}_n$ is the set of length-$n$ prefixes of $L$.
Hence we get the following:
$$\sprad{\mathbf{M}(\calB, \alpha)}= \lim_{n \to \infty}\left( \sum_{w\in L^{pre}_n} \left( \frac{1}{k^n}\right)^{\alpha} \right)^{1/n}$$
and taking the logarithm base $k$ of the equation  $\sprad{\mathbf{M}(\calB, \alpha)} =1$, we obtain this equality:
$$\lim_{n \to \infty}\frac{1}{n}\log_k\left( \frac{1}{k^{n\alpha}} \vert L^{pre} \vert_n \right) =0.$$
Consequently, we obtain the following:
$$\lim_{n \to \infty}\frac{1}{n}\log_k (k^{-n\alpha}) + \lim_{n \to \infty}\frac{1}{n}\log_k(\vert L^{pre} \vert_n ) =0$$
which yields the final equality $\alpha = \lim_{n \to \infty}\frac{\log_k(\vert L^{pre} \vert_n)}{n} =\frac{1}{ \log(k)}h(\calB)$, as desired.

Finally, using Lemma \ref{hdim} we know that $d_H(X)$ is the maximum of the Hausdorff dimensions of the strongly connected components of $\calA$.
We note that by Fact \ref{language_radius_entropy}, the same holds for entropy, because when we put the incidence matrix for $\calA$ in Jordan canonical form $PJP^{-1}$, the blocks of the block-diagonal matrix $J$ correspond to the strongly connected components of $\calA$.
Since $\sprad J$ is the maximum of that of its diagonal blocks, we can conclude that $d_H(X) = \frac{1}{\log (k)}h(X)= d_B(X)$ by Lemma \ref{box_counting_entropy}.

\end{proof}
\end{thm}

\section{Cycle languages and when dimensions disagree}\label{dimensions}

We would like to know in the general case when the Hausdorff dimension of a $k$-automatic set is not equal to its box-counting dimension. 
A fundamental example of this is the set of dyadic rationals in $[0, 1]$: $\{\frac{a}{2^n} : a, n \in \mathbb{N}\} \cap [0, 1]$. 
This set is $2$-automatic because it can be equivalently phrased as the set of numbers in $[0, 1]$ whose binary expansions only have a finite number of one of the two bits ($0$ or $1$). 
It has Hausdorff dimension $0$, as it is countable, and box-counting dimension $1$, as it is dense in the interval. 
Examining a B\"uchi automaton for this set, we note that it seems the reason for this disparity in dimension is that there are many ways to get from the start state to an accept state but few ways (one way, in fact) to loop from an accept state back to itself. 
We can in fact formalize this notion and obtain a sufficient and necessary condition for the equivalence of the two notions of dimension by defining the notion of a \textit{cycle language.}

\begin{defn}\label{cycle}
	In a finite or B\"uchi automaton $\mathcal{A}$ with $q$ as one of its states, the \textit{cycle language} $C_q(\mathcal{A}) \subseteq \Sigma^*$ contains all strings $w$ for which there is a run of $\mathcal{A}$ from state $q$ to itself on the string $w$.
	Let $\calC_q(\calA)$ denote the automaton constructed by taking $\mathcal{A}$, making state $q$ the only start and accept state, and trimming the resulting automaton. 
	Call this the \emph{cycle automaton}.
	
\end{defn}

Note that cycle languages are regular, and this is witnessed by the cycle automaton. 
Note also that because looping to an accept state multiple times (or zero times) is still a loop, $C_q(\mathcal{A})^* = C_q(\mathcal{A})$.

\begin{lem}
\label{hausdorff_dimension_cycle_languages}
	Let $\mathcal{A} = (Q, \Sigma, \delta, S, F)$ be a trim B\"uchi automaton, and let $X = V_k(\mathcal{A})$. 
	Let $X_q = V_k(\calC_q(\mathcal{A}))$ for each $q \in Q$. Then:
	
	\begin{enumerate}[(i)]
		\item $d_H(X) = \max_{q \in F} d_H(X_q)$;
		\item $d_B(X) = \max_{q \in Q} d_H(X_q)$.
	\end{enumerate}
\end{lem}
\begin{proof}
	\begin{enumerate}[(i)]
		
		\item Let $L_q$ be the language of words that have a path from a start state of $\mathcal{A}$ to state $q$. 
		An infinite string is accepted by $\mathcal{A}$ precisely when it has a path that runs from a start state to an accept state and then cycles back to that accept state infinitely often. 
		Thus,
		
		$$L(\mathcal{A}) = \bigcup_{q \in F} \bigcup_{w \in L_q} w C_i(\mathcal{A})^\omega.$$
		
		Let $T_w$ be the (linear) transformation on a set corresponding to prefixing the string $w$. 
		Applying $d_H \circ V_k$ to both sides, we get:
		
		$$d_H(X) = d_H\left(\bigcup_{q \in F} \bigcup_{w \in L_q} T_w(X_q) \right)$$
		
		Applying the formula for Hausdorff dimension of a countable union:
		
		$$d_H(X) = \sup_{q \in F} \sup_{w \in L_q} d_H(T_w(X_q))$$
		
		Because (invertible) linear transformations do not affect dimension:
		
		$$d_H(X) = \sup_{q \in F} \sup_{w \in L_q} d_H(X_q)$$
		
		Eliminating the now-useless quantification over $w$:
		
		$$d_H(X) = \sup_{q \in F} d_H(X_q)$$
		
		Because $A$ is finite:
		
		$$d_H(X) = \max_{q \in F} d_H(X_q).$$
		
		\item Because $\mathcal{A}$ is trim, we have $\overline{X} = V_k(\overline{\mathcal{A}})$ according to Lemma \ref{automaton_closure}. 
		So applying (i), we get that $d_H(\overline{X}) = \max_{q \in Q} d_H(X_q)$ (note that the cycle language $C_q(\mathcal{A})$ does not depend on which states are accepting).
		Yet we know from Theorem \ref{closed_equality} that $d_H(\overline{X}) = d_B(\overline{X}) = d_B(X)$.
		
	\end{enumerate}
\end{proof}

We do immediately get as a corollary that $d_H(X) < d_B(X)$ when $d_H(X_q)$ is larger for some $q \notin F$ than for all $q \in F$. However, this is not a very useful version of the characterization as it stands, because the Hausdorff dimension of $\nu_k(C_q(\mathcal{A})^\omega)$ is not an easy value to compute \textit{a priori.} 
The rest of this section will be focused on reducing the above to an easier problem.

The main step in the process of simplifying the result of Lemma \ref{hausdorff_dimension_cycle_languages} is the following result:

\begin{lem}
\label{hausdorff_omega_limit}
	Let $\mathcal{A}$ be a finite automaton, not necessarily deterministic. 
	Let $L'$ be the $\omega$-language recognized by $\mathcal{A}$ as a B\"uchi automaton, and let $L$ be the language recognized by $\mathcal{A}$ as a finite automaton. 
	Let $X = \nu_k(L')$ and $Y = \nu_k(L^\omega)$. Then $d_H(\overline{X}) \leq d_H(Y)$.
\end{lem}
\begin{proof}
	Without loss of generality, assume $\mathcal{A}$ is finite-trim. 
	Let $\mathcal{B}$ be the induced sub-automaton containing states in $\mathcal{A}$ from which there are arbitrarily long paths to accept states. 
	We note that $\mathcal{B}$ is trim and is equivalent as a B\"uchi automaton to $\mathcal{A}$. 
	Assume that the desired lemma holds for $\mathcal{B}$; let $L_\mathcal{B}$ be the language recognized by $\mathcal{B}$ as a finite automaton, and let $Z = \nu_k(L_\mathcal{B}^\omega)$. 
	Then $L_\mathcal{B} \subseteq L$, hence $d_H(Z) \leq d_H(Y)$. 
	So $d_H(\overline{X}) \leq d_H(Y)$ as well; because $\mathcal{A}$ and $\mathcal{B}$ are equivalent as B\"uchi automata, the lemma then holds for $\mathcal{A}$. 
	Thus we have reduced the lemma to the case where the automaton in question is trim.
	
	Starting over, assume $\mathcal{A}$ is trim. 
	Then $\overline{\mathcal{A}}$, as a finite automaton, recognizes the language $M$ of prefixes of $L$. 
	Let $M'$ be the language $\overline{\mathcal{A}}$ recognizes as a B\"uchi automaton.
	
	By the proof of Lemma \ref{automaton_closure}, the language $M'$ is closed, in the sense that if infinitely many of an infinite string $w$'s prefixes are prefixes of strings in $M'$, then $w \in M'$. 
	Therefore $M' = \vec{M}$, as $M$ is the language of prefixes of strings in $M'$.
    Note furthermore that $L' \subseteq M'$. 
    Moreover, any prefix of a string in $M'$ is a string in $M$ and thus a prefix of a string in $L$, and thus, from trimness, a prefix of a string in $L'$; so $\nu_k(\vec{M}) = \overline{X}$.
	
	Let $Q$ be the state set of $\mathcal{A}$ (and $\overline{\mathcal{A}}$). 
	Choose for each $q \in Q$ a string $u_q$ on which $\mathcal{A}$ runs from a start state to state $q$, and choose a string $v_q$ on which $\mathcal{A}$ runs from state $q$ to an accept state. 
	Note that $u_q v_q \in L$. 
	Let $\ell$ be the maximum length of $v_q u_q$ for $q \in Q$.
	
	Observe that because $\overline{X} = \nu_k(\vec{M})$ is closed, and $M$ is closed under prefixes, we have that $d_H(\overline{X}) = \frac{1}{\log k} h(M)$ by Theorem \ref{closed_equality}. 
	Define the language $M_n$ for each positive integer $n$ as follows: for each string in $M$, and for every accepting run of this string in $\mathcal{A}$, after every $n$ characters, we find the index $q$ of the state $\mathcal{A}$ is in and insert $v_q u_q$; the resulting strings are the elements of $M_n$. 
	We note that $M_n$ and $X_n = \nu_k(\vec{M}_n)$ have the following properties:
	
	\begin{itemize}
		\item $\vec{M}_n$ is a subset of $L^\omega$. 
		A string in $\vec{M}_n$ must have the form $w_1 v_{q_1} u_{q_1} w_2 v_{q_2} u_{q_2} w_3 \dots$ with each $w_j$ a string of length $n$. 
		Note that $w_1 v_{q_1} \in L$, and $u_{q_j} w_{j+1} v_{q_{j+1}}$ is in $L$ as well. 
		So this string is in $L^\omega$, and hence $X_n \subseteq Y$.
		\item The set $X_n$ is still closed; observe that by considering every possible path through $L$ of length $n$ and connecting them with strings of states recognizing $v_q u_q$, it is possible to create a closed B\"uchi automaton for $\vec{M}_n$ and apply Lemma \ref{automaton_closure} again. 
		Therefore, $d_H(X_n) = \frac{1}{\log k} h(M_n^{pre})$ by Theorem \ref{closed_equality}.
		
		\item Let $r$ be a positive integer; every string of length at most $rn$ in $M$ has a corresponding string of length at most $r(n+\ell)$ in $M_n^{pre}$. 
		Thus $|M_n^{pre}|_{\leq r(n+\ell)} \geq |M|_{\leq rn}$, and so:
		
		$$h(M_n^{pre}) = \limsup_{r \to \infty} \frac{\log |M_n^{pre}|_{\leq r(n+\ell)}}{r(n+\ell)}$$
		$$\geq \limsup_{r \to \infty} \frac{\log |M|_{\leq rn}}{r(n+\ell)}$$
		$$= \frac{n}{n+\ell} \limsup_{r \to \infty} \frac{\log |M|_{\leq rn}}{rn}.$$
		
		Now, without loss of generality, assume our alphabet does not contain the character $\$$, and let $N = M\$^*$. 
		Observe that $N$ is still prefix-closed and that $|M|_{\leq n} = |N|_n$ for all $n$. Then the above expression is equal to $\frac{n}{n+\ell} h(N)$, because from Lemma \ref{language_entropy_prefix_limit_summation}, we know that the limit superior defining $h(N)$ is a limit (thus it does not matter if we take a subsequence of the indices).
		
		We would then like to replace $N$ with $M$ in this equation, hence we show $h(N) = h(M)$. 
		Note that because $\overline{\mathcal{A}}$ is trim, every string in $M$ can be extended to a longer string in $M$; and because $M$ is prefix-closed, said longer string can have any length. 
		Therefore $|M|_n$ is monotone in $n$, and so $|N|_n = |M|_{\leq n} \leq (n+1) |M|_n$. Conversely we trivially have $|M|_n \leq |N|_n$. 
		Since multiplication by a linear factor does not change the entropy, we have $h(N) = h(M)$ as required.
		
		\item Let $r$ be a positive integer. 
		Each string in $M$ corresponds to at most $|Q|(\ell+1)$ strings in $M_n^{pre}$ because, at worst, the length of the original string is a multiple of $n$, and thus $M_n^{pre}$ has $|Q|(\ell+1)$ corresponding strings depending on which state the string ends in (which could, due to nondeterminism, be any of them) and how much of the last $v_q u_q$ is added. 
		These corresponding strings are never shorter than the original string; therefore, $|M_n^{pre}|_{\leq r} \leq |Q|(\ell + 1) |M|_{\leq r}$. Thus:
		
		$$h(M_n^{pre}) = \limsup_{r \to \infty} \frac{\log |M_n^{pre}|_{\leq r}}{r}$$
		$$\leq \limsup_{r \to \infty} \frac{\log (|Q|(\ell + 1) |M|_{\leq r})}{r}$$
		$$= \limsup_{r \to \infty} \frac{\log |Q|(\ell + 1) + \log |M|_{\leq r}}{r}$$
		$$= \limsup_{r \to \infty} \frac{\log |Q|(\ell + 1)}{r} + \limsup_{r \to \infty} \frac{\log |M|_{\leq r}}{r}$$
		$$= \limsup_{r \to \infty} \frac{\log |M|_{\leq r}}{r} = h(M).$$
	\end{itemize}
	
	We thus have a collection of subsets $X_n$ of $Y$ such that:
	
	$$\frac{n}{n+\ell} d_H(\overline{X}) = \frac{1}{\log k} \frac{n}{n+\ell} h(M) \leq \frac{1}{\log k} h(M_n^{pre}) = d_H(X_n) \leq \frac{1}{\log k} h(M) = d_H(\overline{X}).$$
	
	Because $\frac{n}{n + \ell} \to 1$ as $n \to \infty$, we conclude that $d_H(X_n) \to d_H(\overline{X})$. 
	Therefore, $d_H(Y) \geq \sup d_H(X_n) \geq d_H(\overline{X})$, as required.
\end{proof}

\begin{cor}
\label{dimension_equiv_cycle}
	Let $\mathcal{A}$ be a B\"uchi automaton recognizing an $\omega$-language $L'$. 
	Assume that $\mathcal{A}$ is trim and has one accept state and one start state, which are the same state. 
	Let $L$ be the language recognized by $\mathcal{A}$ as a finite automaton. 
	Then $L' = L^\omega$, and $d_H(\nu_k(L')) = d_B(\nu_k(L'))$.
\end{cor}
\begin{proof}
	Let $X = \nu_k(L')$. From Theorem \ref{closed_equality}, we know $d_H(\overline{X}) = d_B(\overline{X}) = d_B(X)$, so it suffices to show $d_H(X) = d_H(\overline{X})$; in fact, we need only show $d_H(\overline{X}) \leq d_H(X)$.
	
	Note that if an infinite string $w$ belongs to $L'$, it must have a run that starts at the single start/accept state and revisit it infinitely many times; hence, $w$ can be broken into infinitely many substrings, each of which starts and ends at this state. 
	So $L' = L^\omega$. The second result then follows from Lemma \ref{hausdorff_omega_limit}.
\end{proof}

\begin{lem}
\label{dimension_entropy_cycle}
    Let $L = C_q(\mathcal{A})$ be a cycle language for a B\"uchi automaton. 
    Then $d_H(\nu_k(L^\omega)) = \frac{1}{\log k} h(L)$.
\end{lem}
\begin{proof}
    Let $L'$ be the language $\calC_q(\mathcal{A})$ accepts as a B\"uchi automaton. 
    Then by Corollary \ref{dimension_equiv_cycle}, it suffices to show $d_B(\nu_k(L')) = \frac{1}{\log k} h(L)$. 
    Moreover, by Lemma \ref{box_counting_entropy}, it suffices to show $h((L')^{pre}) = h(L)$, where $(L')^{pre}$ is the language of prefixes of strings in $L'$.
    
    We claim that $(L')^{pre} = L^{pre}$. 
    We have $L^{pre} \subseteq (L')^{pre}$ because $\calC_q(\mathcal{A})$ is trim, so any string that can be extended to a string in $L^{pre}$ can be further extended to an infinite string in $(L')^{pre}$. 
    We have $(L')^{pre} \subseteq L^{pre}$ because any string in $L'$ must have infinitely many prefixes in $L$, so we may extend a string in $(L')^{pre}$ to an infinite string and then cut it off at a sufficiently late occurrence of an accept state.
    
    So we have reduced the lemma to demonstrating that $h(L) = h(L^{pre})$, which is Fact \ref{language_entropy_summation}.
\end{proof}

Combining the above with Lemma \ref{hausdorff_dimension_cycle_languages} gives us the theorem:

\begin{thm}
\label{hausdorff_dimension_cycle_language_entropy}
    Let $\mathcal{A} = (Q, \Sigma, \delta, S, F)$ be a trim B\"uchi automaton, and let $X = V_k(\mathcal{A})$. Let $F$ be the set of indices of accept states in $\mathcal{A}$. Then:
	
	\begin{enumerate}[(i)]
		\item $d_H(X) = \frac{1}{\log k} \max_{q \in F} h(C_q(\mathcal{A}))$;
		\item $d_B(X) = \frac{1}{\log k} \max_{q \in Q} h(C_q(\mathcal{A}))$.
	\end{enumerate}
\end{thm}

\begin{cor}
Let $\calA$ be a trim B\"uchi automaton such that $V_k(\calA) = X$. 
Then $d_H(X)<d_B(X) = h(X)$ if and only if there exists a non-final state $q \in Q \setminus F$ such that for cofinitely many $m \in \N$ there exists $n_m \in \N$ such that for each $f \in F$, the ratio of $|C_f(\calA)|_m$, which denotes the number of paths from $f$ to itself of length $m$, to $k^m$ is strictly less than than the ratio of $|C_q(\calA)|_{n_m}$ to $k^{n_m}$.
\begin{proof}
For the forward implication, we will illustrate the contrapositive; suppose that for every non-final state $q \in Q \setminus F$ there is a final state $f \in F$ such that $\limsup_{n \to \infty} |\{ \sigma:q \to q \: | \: \sigma \text{ is a path of length }n\}| \leq \limsup_{n \to \infty} |\{ \sigma:f \to f \: | \: \sigma \text{ is a path of length }n\}|$.
If this is the case, we conclude that for all $q \in Q$ we have $h(C_q(\calA))\leq h(C_f(\calA))$ for some $f \in F$, where $C_q(\calA)$ is the cycle language of $q$ as defined in \ref{cycle}.
This follows because taking logarithms commutes with $\limsup$.
By Theorem \ref{hausdorff_dimension_cycle_language_entropy}, we conclude that $d_B(X)=d_H(X)$, proving the contrapositive.

For the backwards implication, we suppose that there does exist some $q \in Q \setminus F$ such that such that for cofinitely many $m \in \N$ there exists $n_m \in \N$ such that $\frac{|C_f(\calA)|_m}{k^m}$ (the ratio of paths from $f$ to itself of length $m$ to $k^m$) is strictly less than than $\frac{|C_q(\calA)|_{n_m}}{k^{n_m}}$ (the ratio of paths from $q$ to itself of length $n_m$ to $k^{n_m}$) for each $f \in F$.
Then it is evident that the sequence $(n_m)_{m \in \N}$ witnesses that for all $f \in F$ we have $\limsup_{n \to \infty}\frac{|C_f|_n}{k^n} < \limsup_{n \to \infty}\frac{|C_q(\calA)|_n}{k^n}$.
Taking the logarithm of each side, we conclude $h(C_f(\calA))<h(C_q(\calA))$ for all $f \in F$.
Hence by Theorem \ref{hausdorff_dimension_cycle_language_entropy}, we get that $d_H(X) < d_B(X)$.
\end{proof}
\end{cor}

\section{Hausdorff measure}\label{measure}

As mentioned in Section \ref{Hdim}, the Hausdorff dimension of a fractal $X$ is defined by considering the $s$-dimensional Hausdorff \textit{measure} $\mu_H^s(X)$ for different values of $s$, and in particular, $\mu_H^s(X)$ is $\infty$ for $s < d_H(X)$ and $0$ for $s > d_H(X)$. 
In this section, we will give methods for determining the $d_H(X)$-dimensional Hausdorff measure of various types of $k$-automatic fractal $X$ and make some observations that result from these methods. 
Note that this measure may be zero, a positive real number, or infinite.

We begin by leveraging former work of Merzenich and Staiger. 
The following comprises Lemma 15 and Procedure 1 of \cite{MS94}:

\begin{fact}
	\label{hausdorff_measure_sc_closed}
	Let $\mathcal{A}$ be a strongly connected deterministic B\"uchi automaton such that $X = V_k(\mathcal{A})$ is closed. 
	Then there is exactly one vector $\vec{u}$ such that $\vec{u}$ is an eigenvector corresponding to an eigenvalue of maximum magnitude of $\mathbf{M}(\mathcal{A}, 0)$ and such that $\vec{u}$ contains nonnegative entries, the maximum of which is $1$. 
	The Hausdorff measure $\mu_H^{d_H(X)}(X)$ is the entry in $\vec{u}$ corresponding to the start state of $\mathcal{A}$.
\end{fact}

Note that in particular, this implies $0 \leq \mu_H^{d_H(X)}(X) \leq 1$ in this case.

The next step in this analysis is to extend this work to the case of a general strongly connected automaton. 
In particular, we might rightfully suspect that the Hausdorff dimension and measure of a set recognized by a strongly connected B\"uchi automaton are the same as those of the set's closure. 
It is most convenient to show this first for the dimensions and then for the measures.

\begin{lem}
\label{strongly_connected_hausdorff_closure}
    If $\calA$ is a strongly connected B\"uchi automaton, then $d_H(V_k(\calA)) = d_H(\overline{V_k(\calA)})$.
\end{lem}
\begin{proof}
    Note that removing accept states from $\calA$ can only make the resulting Hausdorff dimension lower, so it suffices to consider the case where $\calA$ has only one accept state. 
    Similarly, by Lemma \ref{strongly_connected_start_state}, it suffices to assume $\calA$ has only one start state, which is also its accept state.
    
    In this case, $\calA$ satisfies the conditions of Corollary \ref{dimension_equiv_cycle}; so $d_H(V_k(\calA)) = d_B(V_k(\calA))$. 
    This is then equal to $d_B(\overline{V_k(\calA)})$, which is in turn equal to $d_H(\overline{V_k(\calA)})$ by Theorem \ref{closed_equality}, because $\overline{V_k(\calA)}$ is a closed $k$-automatic set.
\end{proof}

We show the analogous result for Hausdorff measures by first noting that when such an $X$ is embedded in its closure, the set-difference has lower dimension and hence is null in the higher-dimensional measure.

\begin{lem}
	\label{hausdorff_dimension_closure_minus_sc}
	Let $\mathcal{A}$ be a strongly connected B\"uchi automaton, and let $X = V_k(\mathcal{A})$. 
	Then $d_H(\overline{X} \setminus X) < d_H(X)$ assuming the latter is positive.
\end{lem}
\begin{proof}
	First, we let $\overline{L}$ be the $\omega$-language accepted by $\overline{\calA}$.
	Note that $\nu_k(\overline{L}) = \overline{X}$, so an element of $\overline{X} \setminus X$ must correspond to an infinite string on which $\mathcal{A}$ passes through finitely many accept states before infinitely passing through exclusively non-accepting states. 
	Let $L'$ be the language of strings for which every infinite run is of this form (and for which there is at least one such infinite run). 
	This correspondence is not exact; it is possible that if an element of $\overline{X} \setminus X$ is $k$-rational, one of its representations is in $L'$ while the other is not. 
	This will not affect an analysis of Hausdorff dimension, because (nonzero) Hausdorff dimension does not change based on the membership of a countable set. 
	Our goal is then to show that $d_H(\nu_k(L')) < d_H(X)$.
	
	Let $F^c = Q \setminus F$ be the set of non-accepting states of $\mathcal{A}$, and let $L_q$ for $q \in F^c$ be the set of infinite strings for which every infinite run starting at state $q$ only passes through non-accepting states in $\mathcal{A}$ (and for which there is at least one such infinite run). 
	Since every string in $L'$ has a tail in $L_q$ for some $q$, we know $\nu_k(L')$ is a countable union of scaled copies of $\nu_k(L_q)$. 
	So it suffices to show that $d_H(\nu_k(L_q)) < d_H(\overline{X})$ for all $q \in F^c$. 
	We will show a stronger statement, that $d_H(\overline{\nu_k(L_q)}) < d_H(\overline{X})$. By Theorem \ref{closed_equality}, it suffices to show the corresponding entropy statement, that $h(L_q^{pre}) < h(\overline{L}^{pre})$.
	
	Let $\log \alpha = h(L_q^{pre})$. 
	By Lemma \ref{language_entropy_prefix_limit}, $\log \alpha = \lim_{n \to \infty} \frac{\log |L_q^{pre}|_n}{n}$. 
	For every non-accepting state $q'$ of $\mathcal{A}$, there exists a string on which $\mathcal{A}$ cycles from state $q'$ to itself while passing through an accept state. Choose such a cycle for each $q' \in F^c$, and let $m$ be the least common multiple of their lengths; by repeating the cycles if necessary, we may assume they all have the same length $m$.
	
	Let $M$ be the $\omega$-language defined as follows: every string in $M$ is made up of blocks $m$ characters long, with each block a ``normal block'' or a ``cycle block.'' 
	The normal blocks, when taken together with the cycle blocks removed, must form a string in $L_q$. 
	The cycle blocks are each one of the cycles of length $m$ mentioned above, whichever one corresponds to the state of $\mathcal{A}$ the string is in at the time, in a run witnessing that the normal blocks form a string in $L_q$. 
	Note that when every string in $M$ is prefixed by a constant string that corresponds to a path from the start state of $\mathcal{A}$ to state $q$ (which does not affect entropy), the result is a subset of $\overline{L}$. 
	So $h(M^{pre}) \leq h(\overline{L}^{pre})$, and it suffices to show $h(L_q^{pre}) < h(M^{pre})$.
	
	Now for every string in $L_q^{pre}$, choose a run that witnesses this. 
	Let $k$ be a positive integer, and let $0 \leq r \leq k$. 
	For any string in $L_q^{pre}$ of length $rm$, we may apply the chosen run and insert $(k-r)$ cycle blocks at any choice of block boundaries in order to produce a string in $M^{pre}$ of length $km$. 
	Therefore:

	$$|M^{pre}|_{km} \geq \sum_{r=0}^k {k \choose r} |L_q^{pre}|_{rm}.$$
	
	Let $1 < \beta < \alpha$. 
	One can check that eventually $|L_q^{pre}|_n > \beta^n$. 
	So for sufficiently large $k$:
	
	$$|M^{pre}|_{km} > \sum_{r=0}^k {k \choose r} \beta^{rm} = (\beta^m + 1)^k.$$
	
	So:
	
	$$h(M^{pre}) = \limsup_{n \to \infty} \frac{\log |M^{pre}|_n}{n}$$
	$$\geq \limsup_{k \to \infty} \frac{\log |M^{pre}|_{km}}{km}$$
	$$\geq \limsup_{k \to \infty} \frac{\log \left((\beta^m + 1)^k\right)}{km}$$
	$$= \limsup_{k \to \infty} \frac{k \log \left(\beta^m + 1\right)}{km}$$
	$$= \frac{\log \left(\beta^m + 1\right)}{m}.$$
	
	So for all $1 < \beta < \alpha$, we have $h(M^{pre}) \geq \frac{\log \left(\beta^m + 1\right)}{m}$. 
	Thus:
	
	$$h(M^{pre}) \geq \lim_{\beta \nearrow \alpha} \frac{\log \left(\beta^m + 1\right)}{m} = \frac{\log \left(\alpha^m + 1\right)}{m} > \frac{\log \left(\alpha^m\right)}{m} = \log \alpha = h(L_q^{pre}).$$
	
	This concludes the proof.
\end{proof}

\begin{cor}
\label{hausdorff_measure_closure}
	Let $\mathcal{A}$ be a strongly connected B\"uchi automaton, and let $X = V_k(\mathcal{A})$ have Hausdorff dimension $d_H(X) = \alpha$. 
	Then $\mu_H^{\alpha}(X) = \mu_H^{\alpha}(\overline{X})$.
\end{cor}
\begin{proof}
    We examine two cases. 
    If $\alpha$ is positive, then the previous lemma gives us $d_H(\overline{X} \setminus X) < \alpha$, hence $\mu_H^\alpha(\overline{X} \setminus X) = 0$. 
    The result then follows from additivity of Hausdorff measure.
    
    Assume on the other hand that $\alpha = 0$; then $\mu_H^\alpha$ is the counting measure.
    If $X$ is infinite, then so is $\overline{X}$, hence the counting measure agrees for both.
    If $X$ is finite, then $\overline{X} = X$, as any finite set in a metric space is closed. So in either case $\mu_H^0(X) = \mu_H^0(\overline{X})$.
\end{proof}

Our goal is now to extend this result and compute the Hausdorff measure of any $k$-automatic fractal. 
We will do this by making use of a class of B\"uchi automata called \textit{unambiguous} B\"uchi automata:

\begin{defn}
\label{unambiguous_buchi_automata}
	Let $\mathcal{A}$ be a B\"uchi automaton. We say that $\mathcal{A}$ is \textit{unambiguous} if for every infinite string $w$ accepted by $\mathcal{A}$, there is exactly one accepting run of $\mathcal{A}$ on $w$.
\end{defn}

Note that any deterministic B\"uchi automaton is unambiguous. 
But in the case of B\"uchi automata, we know that nondeterminism is sometimes necessary in order to achieve full computational capability. 
Fortunately, a result of \cite{CM03} tells us that this capability can still be fully realized in the unambiguous case:

\begin{fact}
\label{unambiguous_equivalence}
	Let $L$ be a regular $\omega$-language. 
	Then there is an unambiguous B\"uchi automaton $\mathcal{A}$ with $L(\mathcal{A}) = L$.
	Moreover, given any B\"uchi automaton for $L$, we can effectively convert it to an unambiguous automaton.
\end{fact}

We will now extend Corollary \ref{hausdorff_measure_closure}, giving us a method to compute the Hausdorff measure of any $k$-automatic fractal. 
Using the previous fact, we may assume that the B\"uchi automaton defining the fractal is unambiguous. 
Next, it is useful for us to define functions associating each infinite string in an $\omega$-language $L$ with its \textit{key state} and \textit{key prefix}:

\begin{defn}
	Let $\mathcal{A}$ be an unambiguous B\"uchi automaton with set of states $Q$ and recognizing an $\omega$-language $L$. 
	The function $k_\mathcal{A} : L \to Q$ maps a string $w$ to the unique \textit{key state} $q$ such that when $\mathcal{A}$ performs an accepting run on $w$:
	
	\begin{itemize}
		\item the run passes through state $q$;
		\item after first passing through state $q$, the run never leaves its strongly connected component;
		\item before first passing through state $q$, the run never enters its strongly connected component.
	\end{itemize}

	The function $p_\mathcal{A} : L \to \Sigma^*$ maps $w$ to the \textit{key prefix} of $w$ running from the start state of $\mathcal{A}$ to the first occurrence of $k_\mathcal{A}(w)$.
\end{defn}

Note that when a run leaves a strongly connected component, it can never return to it, or else the states in between would also be a part of the same component, a contradiction. 
So $k_\mathcal{A}$ and $p_\mathcal{A}$ are well-defined functions.

\begin{lem}
\label{hausdorff_measure_one_key_state}
	Let $\mathcal{A}$ be an unambiguous B\"uchi automaton with set of states $Q$ and recognizing an $\omega$-language $L$.
	Let $Q' \subseteq Q$ be the set of states whose strongly connected component contains an accept state. 
	For each $q \in Q'$, let $\mathcal{A}_q$ be the automaton created by moving the start state of $\mathcal{A}$ to $q$ and removing all transitions out of its strongly connected component, and let $L_q$ be the $\omega$-language it accepts. 
	Let $M_q$ be the $\omega$-language containing all strings $w \in L$ with $k_\mathcal{A}(w) = q$. Then $d_H(\nu_k(M_q)) = d_H(\nu_k(L_q))$, assuming $M_q \neq \varnothing$. 
	Moreover, let $P_q = \{p_\mathcal{A}(w) : w \in M_q\}$; then for any $\alpha$:
	
	$$\mu_H^{\alpha}(M_q) = \sum_{u \in P_q} \frac{\mu_H^{\alpha}(L_q)}{k^{\alpha|u|}}$$
	
	(here $|u|$ is the length of $u$).
\end{lem}
\begin{proof}
	Let $M_{q,u}$ be the set of $w \in L$ with $k_\mathcal{A}(w) = q$ and $p_\mathcal{A}(w) = u$. 
	Such strings decompose as $w = uv$, where $v$ is a string that, starting from state $q$, passes through accept states in its strongly connected component infinitely often without leaving the component. 
	In other words, $M_{q,u} = u L_q$. 
	It follows that $\nu_k(M_{q,u})$ is an affine copy of $\nu_k(L_q)$, which is scaled down with a factor of $f = \frac{1}{k^{|u|}}$. 
	By properties of Hausdorff measure, $\mu_H^{\alpha}(M_{q,u}) = \frac{\mu_H^{\alpha}(L_q)}{k^{\alpha|u|}}$.
	
	Note then that $M_q = \bigcup_{u \in P_q} M_{q,u}$, because $M_{q,u}$ is empty for $u \notin P_q$. 
	Since $P_q \subseteq \Sigma^*$ is a countable set, we may apply countable additivity to get the desired formula.
\end{proof}

This process may then be applied to every key state to produce the following result:

\begin{thm}
\label{hausdorff_measure_all_key_states}
	Let $\mathcal{A}$, $Q$, $Q'$, $L$, $L_q$, $M_q$ be as in the previous lemma. Then:
	
	\begin{enumerate}[(i)]
		\item $d_H(\nu_k(L)) = \max_{q \in Q'} d_H(\nu_k(L_q))$,
		\item with $\alpha = d_H(\nu_k(L))$, $\mu_H^{\alpha}(L) = \sum_{q \in Q'} \mu_H^{\alpha}(M_q)$.
	\end{enumerate}
\end{thm}
\begin{proof}
	For (i), we need only note that from the previous lemma $d_H(\nu_k(L_q)) = d_H(\nu_k(M_q))$ and that $L = \bigcup_{q \in Q'} M_q$. Then (ii) also follows from $L = \bigcup_{q \in Q} M_q$.
\end{proof}

\section{Model-theoretic consequences}\label{models}

For the basics of first-order logic and model theory, see \cite{M02}, from which we also take our notation and conventions.
We recall the following correspondence between subsets of $\R$ recognized by B\"uchi automata and definable subsets of the first-order structure $\calR_k=(\R, <, +, \Z, X_k(x,u,d))$ where the ternary predicate $X_k(x,u,d)$ holds precisely if $u$ is an integer power of $k$, and in the $k$-ary expansion of $x$ the digit specified by $u$ is $d$, i.e. $d$ is the coefficient of $u$ in the expansion.
In particular, note that $u \in k^{\Z}$ and $d \in \{0,\ldots ,k-1\}$ in any tuple $(x,u,d)$ such that $X_k(x,u,d)$ holds.
\begin{thm}[\cite{BRW98}]
For each $n \in \N$, any $A \subseteq \R^n$ is $k$-automatic if and only if $A$ is definable in $\calR_k$.
\end{thm}

In conjunction with results from \cite{BG21}, we obtain the following consequences for definable subsets of the structure $\calR_k$.
Below, by ``Cantor set'' we mean a nonempty compact set that has no isolated points and has no interior.

\begin{lem}\label{loglem}
Suppose that $\calA$ is a trim B\"uchi automaton with $m$ states on alphabet $[k]$ that has one accept state, which is also the start state.
Let $L$ be the language it accepts as a B\"uchi automaton.
Suppose that $w \in [k]^n$ is not the prefix of any word in the language $L$.
Then every subset of $[0,1]$ of the form $(\nu_k(w0^{\omega}),\nu_k(w(k-1)^{\omega}))$, with $w \in [k]^{\ell}$, $\ell \in \N$, has a subinterval of size at least $k^{-(\ell+m+n)}$ disjoint from $\nu_k(L)$.
Moreover, the box-counting dimension of $\nu_k(L)$ is at most $\log_{k^{m+n}}(k^{n+m}-1)<1$.

\begin{proof}
For each $q \in Q$ we let $s_q$ be a minimal path from $q$ back to the start state.
We know a minimal path visits each state at most once. 
For a prefix $w \in [k]^*$ such that a run on $\calA$ stops at state $q$, the interval $\tilde{I} =(\nu_k(ws_qv0^{\omega}),\nu_k(ws_qv(k-1)^{\omega}))$ cannot be in $\nu_k(L)$.
Observe that the length of $\tilde{I}$ is at least $k^{-(\ell+m+n)}$.

For the ``moreover,'' observe that by Lemma \ref{box_counting_entropy} and Lemma \ref{automaton_closure}, we have the following:
$$d_B(\nu_k(L)) = d_B(\overline{\nu_k(L)}) = \frac{1}{\log(k)} h(L^{pre}) = \lim_{d \to \infty} \frac{\log(|L^{pre}|_{d})}{d \log(k)} = \lim_{d \to \infty} \frac{\log(|L^{pre}|_{d(n+m)})}{d(n+m) \log(k)}. $$
We proceed by induction on $d$.
For the base case, note $|L^{pre}|_{n+m} \leq k^m(k^n-1)<k^{n+m}-1$. 
For $d>1$, we know $|L^{pre}|_{d(n+m)} \leq |L^{pre}|_{(d-1)(n+m)}(k^{n+m}-1)$ because for each $w \in  |L^{pre}|_{(d-1)(n+m)}$ there exists $q' \in Q$ such that $ws_{q'}v$ is not in $|L^{pre}|_{d(n+m)}$.
By induction, we get the following:
$$ \lim_{d \to \infty} \frac{\log(|L^{pre}|_{d(n+m)})}{d(n+m)\log(k)} \leq  \lim_{d \to \infty} \frac{\log((k^{n+m}-1)^d)}{d(n+m)\log(k)},$$
and hence $d_B(\nu_k(L)) \leq \log_{k^{n+m}}(k^{n+m}-1)$.
\end{proof}
\end{lem}

Recall that $\mu_H^1$ is the Hausdorff 1-measure, which equals Lebesgue measure on subsets of $\R$.

\begin{lem}\label{measuremin}
Suppose that $X\subseteq [0,1]$ is a $k$-automatic set with Hausdorff dimension 1, and $X$ is recognized by B\"uchi automaton $\calA$.
If $\mu_H^1(X \cap I)>0$ for every interval $I \subseteq [0,1]$ with $k$-rational endpoints, then there is $\epsilon>0$ such that for each such $I$ we have $\mu_H^1(X \cap I) \geq \epsilon \cdot \operatorname{diam}(I)$.
\begin{proof}
Let $m$ be the number of states in $\calA$.
By hypothesis the Hausdorff dimension of $X \cap I$ is 1 for every interval $I \subseteq [0,1]$ with $k$-rational endpoints, and $\mu_H^1(X \cap I)$ is positive.
This means that for each prefix $w$ of $L(\calA)$, there exists a strongly connected component $\calC_q(\calA)$ that contains an accept state $q$ and for which the image under $V_k$ of the set of words with prefix $w$ that are accepted via a run contained in $\calC_q(\calA)$ has positive Hausdorff 1-measure.
In particular, there is a word $v$ with length at most $m+|w|$ such that $v$ extends $w$ and $\mu_H^1(\nu_k(vC_q(A)^{\omega}))>0$, where $C_q(A)$ is the cycle language for $\calC_q(\calA)$, assuming without loss of generality that a run of $v$ terminates at state $q$.
We set $a$ to be the minimum of $\mu_H^1(V_k(C_q(\calA)))$ where $q$ ranges over all states in a strongly connected component of $\calA$ that contains an accept state and has positive Hausdorff 1-measure.
We let $\epsilon = a \cdot k^{-m}$, and conclude that for each $I \subseteq [0,1]$, we have $\mu_H^1(X \cap I) \geq \epsilon \cdot \operatorname{diam}(I)$, as desired.
\end{proof}
\end{lem}

\begin{thm}
Suppose $X \subseteq [0,1]^n$ is $k$-automatic.  
There exists a set $A \subseteq [0,1]$ definable in $(\R, <,+, X)$ such that $d_B(A) \neq d_H(A)$ if and only if either a unary Cantor set is definable in $(\R,<,+,X)$ or a set that is both dense and codense on an interval is definable in $(\R,<,+,X)$.

\begin{proof}
$(\impliedby)$ Suppose that $C$, a unary Cantor set, is definable in the structure $(\R,<,+,X)$.
Let $A$ be the set of endpoints of the complementary intervals (e.g. the endpoints of the middle thirds that are removed in constructing the ternary Cantor set), which form a countable set whose closure is the entire Cantor set.
This follows since each element of a Cantor set can be approximated by endpoints of the complement within $[0,1]$.
We conclude that $d_H(A) = 0$ and $d_B(A) = d_B(\overline{A}) \neq 0$ by Lemma 6.7 in \cite{BG21}.

Now suppose that $D$ is a set that is dense and codense in an interval $J \subseteq [0,1]$, and is definable in $\calR_k$.
Note that if there exists any interval $I \subseteq J$ such that $d_H(D \cap I) <1$, then because $D$ is dense in $J$ and hence also in $I$, we conclude that $d_B(D\cap I) = d_H(\overline{D} \cap I)=1 > d_H(D\cap I)$, and we are done.
Similarly if $d_H(\overline{D} \setminus D \cap I)<1$ for some interval $I \subseteq J$, since $D$ is codense in $J$. 

The only case remaining is the one in which for every interval $I \subseteq J$, we know $d_H(D\cap I)=1$ and $d_H(\overline{D} \setminus D\cap I)=1$.
If this is the case then both $\calA$, the automaton that recognizes $D$, and $\calA_c$, the automaton that recognizes $\overline{D}\setminus D$, each have a strongly connected component containing an accept state whose cycle language has the same Hausdorff dimension as the whole automaton, namely Hausdorff dimension 1.
Suppose that $S_q(\calA)$ is the strongly connected automaton that recognizes the cycle language of state $q$ in $\calA$ that has the same Hausdorff dimension as $D$. 
Similarly suppose that $S_{q'}(\calA_c)$ is the strongly connected automaton that recognizes the cycle language of state $q'$ that has the same Hausdorff dimension as $\overline{D} \setminus D$.

If the image of $S_q(\calA)$ under $V_k$ is nowhere dense in $[0,1]$, then $S_q(\calA)$ satisfies the hypotheses of Lemma \ref{loglem}, since being nowhere dense necessitates that at least one finite word is not a prefix of the language.
Similar logic holds for $S_{q'}(\calA_c)$.
So by Lemma \ref{loglem}, they would have box-counting dimension less than one, and hence their closures would as well. 
This would contradict our assumption that they each respectively witness the Hausdorff dimensions of $\calA$ and $\calA_c$ being one, since Hausdorff dimension is bounded above by box-counting dimension, so $S_q(\calA)$ and $S_{q'}(\calA_c)$ cannot have box-counting dimension less than one.

Therefore $S_q(\calA)$ and $S_{q'}(\calA_c)$ both must recognize somewhere dense sets. 
Hence the closure of each must contain an interval and thus has positive Lebesgue measure.
By Corollary \ref{hausdorff_measure_closure} we know that both $S_q(\calA)$ and $S_{q'}(\calA_c)$ have the same Hausdorff 1-measure as their closures (as B\"uchi automata).
So by Lemma \ref{hausdorff_measure_one_key_state}, both $D$ and $\overline{D}\setminus D$ have positive Lebesgue measure, which equals Hausdorff 1-measure, on every subinterval $I \subseteq J$.
In particular, both $D \cap I$ and $(\overline{D}\setminus D) \cap I$ have positive Lebesgue measure for every $I \subseteq J$.

Fixing one such $I$, we may assume without loss of generality that $\overline{D}$ is an interval.
Applying Lemma \ref{measuremin}, we know that both $D\cap I$ and $\overline{D}\setminus D \cap I$ have positive Lebesgue measure bounded uniformly below by $\epsilon \cdot \operatorname{diam}(I)$ and $\epsilon ' \cdot  \operatorname{diam}(I)$ respectively, for fixed $\epsilon, \epsilon '>0$ given to us by Lemma \ref{measuremin}.
Since $\overline{D}\setminus D \cap I$ is a subset of the complement of $D\cap I$, we conclude $\epsilon' \leq 1-\epsilon$, and also get the following:
$$\epsilon' \cdot \operatorname{diam}(I) \leq \mu_H^1(D \cap I) \leq \epsilon \operatorname{diam}(I).$$
By the Lebesgue density theorem, since $D \cap I$ is dense in $I$ for each $I$ with $k$-rational endpoints, we know that
$$\lim_{\operatorname{diam}(I) \to 0} \frac{\mu_H^1(D \cap I)}{\operatorname{diam}(I)}$$
exists and equals 1.
Yet we observe that $\epsilon' \leq \frac{\mu_H^1(D \cap I)}{\operatorname{diam}(I)} \leq \epsilon$, so we must have $\epsilon = 1$.
But we conclude the same holds for $ \frac{\mu_H^1(\overline{D}\setminus D \cap I)}{\operatorname{diam}(I)}$, so we also need $\epsilon' =1$.
This contradicts the fact that $\epsilon' \leq 1 - \epsilon$, and we are done.

$(\implies)$ Suppose that there is some set $A \subseteq [0,1]$ definable in $(\R,<,+,X)$ such that $d_B(A) \neq d_H(A)$.
Then it must be the case that $A \neq \overline{A}$.
Let $\calA$ be a trim B\"uchi automaton recognizing $A$.
Note that we know $A$ has no interior.
Suppose it is not the case that $A$ is somewhere dense, i.e. $A$ is nowhere dense.
Then, we may pass from $A$ to $\overline{A}$, and since $A$ is nowhere dense we know that $\overline{A}$ is nowhere dense as well.

We know that $d_B(\overline{A})=d_B(A)>0$; otherwise, since $d_H(A) \leq d_B(A)$, they would have to agree on $A$.
If $d_B(\overline{A}) = 1$, then $d_H(\overline{A})=1$ by Theorem \ref{closed_equality}.
This implies that there is a state $q$ in $\overline{\calA}$ such that $d_H(V_k(S_q(\overline{\calA})))=1$ by Lemma \ref{strongly_connected_hausdorff_closure}.
In order for $\overline{A}$ to be nowhere dense, it must be the case that each strongly connected component of $\overline{\calA}$ omits at least one string of finite length as a prefix.
Otherwise, every finite-length string is the prefix of some string in $C_q(\overline{\calA})$.
Since there exists a string $w \in [k]^*$ such that $\nu_k(wC_q(\overline{\calA})^{\omega}) \subseteq \overline{A}$, 
this would imply that $\overline{A}$ is dense in (and in fact, since $\overline{A}$ is closed, contains) the interval $\nu_k(w[k]^*)$, contradicting that $\overline{A}$ is nowhere dense.
So for each state $q$ in $\overline{\calA}$, let $v_q \in [k]^*$ be a word omitted from the prefixes of $C_q(\overline{\calA})$ whose length is minimal among such finite words.
By Lemma \ref{loglem}, we conclude that $V_k(S_q(\overline{\calA}))$ has Hausdorff dimension less than $1$ for all states $q$ in $\overline{\calA}$.
So by Lemma \ref{hdim}, we know that $d_H(\overline{A})$ is the maximum of $d_H(V_k(S_q(\overline{\calA})))$ for all states $q$ in $\overline{\calA}$,
and we conclude that $1>d_H(\overline{A})>0$. 
By \cite{BG21}, there is a definable unary Cantor set in $(\R,<,+,X)$.

In the case that $A$ is somewhere dense, say on interval $I \subseteq [0,1]$, suppose it is not also codense in $I$.
Then some subinterval of $I$ is contained entirely in $A$, so $d_H(A) =1=d_B(A)$, a contradiction.
\end{proof}
\end{thm}

\end{document}